\documentclass[a4paper,11pt,reqno]{amsart}
\usepackage{amsmath}
\usepackage[T1]{fontenc}
\usepackage{amssymb}
\usepackage{amsthm,graphicx}
\usepackage{color}
\usepackage[lmargin=2.5 cm,rmargin=2.5 cm,tmargin=3.5cm,bmargin=2.5cm,
paper=a4paper]{geometry}
\newcommand{\Bk}{\color{black}}

\newcommand{\Be}{\color{blue}}

\newcommand{\Ab}{\mathbf A}
\newcommand{\ab}{\mathbf a}
\newcommand{\af}{\mathfrak a}

\newcommand{\kp}{\kappa}
\newcommand{\Om}{\Omega}

\newcommand{\Fb}{\mathbf F}

\newcommand{\R}{\mathbb R}
\newcommand{\C}{\mathbb C}
\newcommand{\Ea}{\mathrm{E}_{\rm AB}}
\newcommand{\Ee}{\mathrm{E}_{\varepsilon}}
\newcommand{\Eab}{\mathcal E_{{\rm AB}_h}}
\newcommand{\Eef}{\mathcal E_{h,\varepsilon}}
\newcommand{\Eefn}{\mathcal E_{h,\varepsilon_n}}

\newcommand{\fb}{\mathbf{f}}
\newcommand{\Fab}{\mathbf F_{\rm AB}}

\DeclareMathOperator{\curl}{curl}

\newtheorem{thm}{Theorem}[section]
\newtheorem{prop}[thm]{Proposition}
\newtheorem{lem}[thm]{Lemma}
\newtheorem{corol}[thm]{Corollary}

\newtheorem{defn}[thm]{Definition}

\theoremstyle{remark}

\newtheorem{rem}[thm]{Remark}

\newcommand {\p}{\partial}
\newcommand{\q}{\quad}

\newcommand{\eq}{\begin{equation}}
\newcommand{\eeq}{\end{equation}}

\def\curl{\text{\rm curl\,}}
\def\div{\text{\rm div}}

\def\q{\quad}

\def\v{\vskip}

\def\k{\kappa}
\def\lam{\lambda}
\def\p{\partial}

\def\O{\Omega}

\def\var{\varepsilon}

\def\A{\bold A}
\def\C{\bold C}

\def\E{\bold E}

\def\F{\bold F}

\def\R{\Bbb R}

\def\Z{\mathbb Z}
\def\0{\bold 0}

\def\C{\mathbb C}

\numberwithin{equation}{section}

\title[Aharonov-Bohm magnetic fields]
{Oscillatory patterns in the Ginzburg-Landau model driven by the Aharonov-Bohm potential}

\author{Ayman Kachmar}
\author{XingBin Pan}
\address[A. Kachmar]{Department of Mathematics, Lebanese University, Nabatieh, Lebanon.}
\email{ayman.kashmar@gmail.com}

\address[X.B. Pan]{Department of Mathematics,
East China Normal University, and NYU-ECNU Institute of Mathematical Sciences at NYU Shanghai, Shanghai 200062, P.R. China}
\email{xbpan@math.ecnu.edu.cn  }
\date{\today}
\thanks{Mathematics Subject Classification (2010): 35B40, 35P15, 35Q56}
\begin{document}
\begin{abstract}
We consider the  Aharonov-Bohm magnetic potential  and  study the transition from normal to superconducting solutions within the Ginzburg-Landau model of superconductivity. We  obtain oscillations consistent with the Little-Parks effect. We study the same problem but for a regularization of the Aharonov-Bohm potential, which leads to an interesting Aharonov-Bohm like magnetic field, and we prove that the transition between superconducting and normal solutions is not monotone too. Our results show a mechanism to derive the Aharonov-Bohm magnetic potential starting from a step magnetic field thereby presenting a new aspect of magnetic steps, besides their favoring of the celebrated \emph{edge states}.\Bk
\end{abstract}
\maketitle

\tableofcontents

\section{Introduction}


Non-monotone phase transitions, as the one observed in the Little-Parks experiment \cite{LP},  can be explained mathematically by the property of the lack of strong \emph{diamagnetism}. The question whether the lowest eigenvalue of the magnetic Laplacian is a monotone function of the magnetic field strength  had an early appearance in the literature  \cite{E}. Such  monotonicity has been established in several generic situations \cite{A, BF, FH1, FH2, FK-jde, FP-ball} and has been related to the concentration of the ground states under a magnetic field of large intensity.  Examples violating this monotonicity property are usually related to defects of topological nature, like in general tubular domains \cite{HK}. In disc domains,  defects can be produced by  the variation of the magnetic field \cite{FP} or by imposing a   boundary condition with a strong coupling parameter  \cite{KS}. Recently, in \cite{San}, oscillations has been produced by a more general set-up related to the phase space concentration properties of  ground states.

The Aharonov-Bohm potential  induces a topological defect by puncturing  the domain, and it yields  periodic eigenvalues. We regularize the Aharonov-Bohm  potential by considering a natural  approximation of it which does not cause a topological defect. Interestingly, the regularized potential will still produce the same oscillatory behavior driven by the Aharonov-Bohm eigenvalue. We make this observation rigorous within the Ginzburg-Landau model of superconductivity, and as an outcome, we produce an interesting example of a non-monotone phase transition between the superconducting and normal solutions, and at the same time,  we show a new situation of the lack of strong diamagnetism (Theorem~\ref{thm:osi-ms} below). A notable feature of our results is their validity in a general simply connected domain, whereas  the earlier  results do hold  in disc domains.
\Bk

\subsection{The Ginzburg-Landau model}

Ginzburg and Landau introduced a phenomenological model of the response of superconducting materials to applied magnetic fields.
The behavior of the material is described via the critical configurations of the Ginzburg-Landau functional, defined as follows,
\begin{equation}\label{eq:GL*}
\mathcal E[\psi,\mathfrak A]=\int_\Omega \left(|(\nabla-i\mathfrak A)\psi|^2-\kappa^2|\psi|^2+\frac{\kappa^2}2|\psi|^4\right)\,dx+
\int_\Omega|\curl(\mathfrak A-\mathfrak F)|^2\,dx\,,
\end{equation}
where
\begin{itemize}
\item $\Omega\subset\R^2$ represents the horizontal cross section of the superconducting sample\,;
\item $\kappa\in(0,+\infty)$ is the Ginzburg-Landau parameter\,;
\item $\mathfrak F$ is a given vector field, the applied magnetic potential, such that $\curl\mathfrak F$ is the intensity of a vertical applied magnetic field.
\item $(\psi,\mathfrak A)$ represents the superconducting properties of the material as follows:
\begin{itemize}
\item $|\psi|^2$ measures the density of the superconducting electrons\,;
\item $\curl\mathfrak A$ measures the induced magnetic field in the sample.
\end{itemize}
\end{itemize}
In the two dimensional case we denote by $\curl \A=\p_1A_2-\p_2A_1$. The parameter $\kappa$ will be \emph{fixed} throughout this paper. For this reason, we skip it from the notation. On the opposite we will consider the variation of the parameter $h>0$, that we will introduce in order to display  the   intensity of the applied magnetic field as follows. We rescale the Ginzburg-Landau functional in \eqref{eq:GL*}  by writing $\mathfrak A=h\Ab$ and $\mathfrak F=h\Fb$. Hence, we arrive at the new functional
\begin{equation}\label{eq:GL}
\mathcal E_h(\psi, \Ab)=\int_\Omega \left(|(\nabla-ih\Ab)\psi|^2-\kappa^2|\psi|^2+\frac{\kappa^2}2|\psi|^4\right)\,dx+
h^2\int_\Omega|\curl(\Ab-\Fb)|^2\,dx\,.
\end{equation}
The variational space for this functional depends on the nature of the magnetic potential $\Fb$. In fact:
\begin{itemize}
\item If  $\Fb\in H^1(\Omega;\Bbb R^2)$, then $\mathcal E_h(\psi,\Ab)$ is well defined for all $(\psi,\Ab)\in H^1(\Omega;\C)\times H^1(\Omega;\R^2)$\,.
\item If $\Fb\in L^q_{\rm loc}(\R^2;\Bbb R^2)\cap L^2_{\rm loc}(\R^2\setminus\{0\};\Bbb R^2)$ for some $q\in(1,2)$, then $\mathcal E_h(\psi,\Ab)$ is  well defined for all $\psi\in H^1_{h\Fb}(\Omega;\C)$ and $\Ab\in H^1(\Omega;\R^2)+\Fb$, where $\psi\in H^1_{h\Fb}(\Omega;\C)$ means that $(\nabla-ih\Fb)\psi\in  L^2(\Omega;\Bbb C^2)$ and  $\psi\in L^2(\Omega;\Bbb C)$ (see Sec.~\ref{sec:vs} below for the precise definition of this space).
\end{itemize}
 Hereafter the spaces of real-valued functions, complex-valued functions, and real vector-valued functions are denoted by $L^p(\O), L^p(\O;\Bbb C), L^p(\O;\Bbb R^2)$ respectively. However the norms in these spaces are denoted by the same notation $\|\cdot\|_{L^p(\O)}$.

The case of a uniform applied magnetic field, $\curl\Fb=1$, has been extensively studied in the literature (see the two monographs \cite{FH-b,SS-b} and the references therein), particularly in the framework of critical magnetic fields associated with the various phase transitions in the Ginzburg-Landau model. Recently, the analysis of non-uniform applied magnetic fields matches with some interesting physical phenomena like the Little-Parks effect \cite{LP}  and the presence of \emph{edge} states that concentrate on curves \cite{A, AKP, HK-arma,PK}. More precisely, non-uniform magnetic fields could produce defects of topological nature \cite{FP}. We would like to address this kind of behavior by proving that a \emph{large} uniform magnetic field applied on a \emph{small} region of the sample (magnetic step) produces an effective energy involving the  \emph{Aharonov-Bohm} potential (see Theorem~\ref{thm:ms} in this paper); the later energy shows oscillations in the spirit of the Little-Parks effect (see Corollary~\ref{corol:AB} in this paper). Our contribution displays a new example where  normal/superconducting oscillations exist, and at the same, presents a new aspect of \emph{magnetic steps} besides their celebrated feature of  producing \emph{edge states}.\medskip

In this paper, we work under the hypothesis:
\begin{itemize}
\item $\Omega$ is open, bounded,  simply connected domain  and with   a boundary of class $C^2$\,;
\item $0\in\Omega$.
\end{itemize}
We fix $\varepsilon_0>0$ so that $D(0,\varepsilon_0):=\{x\in\R^2,~|x|<\varepsilon_0\}\subset\Omega$.

\subsection{Aharonov-Bohm potential}
This is the vector field
\begin{equation}\label{eq:F-AB}
\Fb_{\rm AB}(x)=\left(\frac{-x_2}{2\pi |x|^2},\frac{x_1}{2\pi |x|^2}\right)\qquad\big(x=(x_1,x_2)\in\R^2\big)\,,
\end{equation}
which satisfies
\begin{equation}\label{eq:F-AB*}
\Fab\in L^p_{\rm loc}(\R^2,\Bbb R^2)\quad \forall\,p\in[1,2) \quad{\rm and}\quad\curl\Fb_{\rm AB}=\delta_0~{\rm in~}\mathcal D'(\R^2)\,.
\end{equation}
We introduce the space
\begin{equation}\label{eq:spaceH-AB}
\mathcal H_{\rm AB}=H^1_{h\Fb_{\rm AB}}(\Omega;\C)\times (H^1(\Omega;\R^2)+\Fb_{\rm AB})\,,
\end{equation}
and the ground state energy
\begin{equation}\label{eq:enAB}
\Ea(h)=\inf\{\Eab(\psi,\Ab)~:~(\psi,\Ab)\in \mathcal H_{\rm AB}\}\,,
\end{equation}
where $\Eab(\psi,\Ab)$ is defined by \eqref{eq:GL} for $\Fb=\Fb_{\rm AB}$, i.e.
\begin{equation}\label{eq:GL-AB}
\Eab(\psi, \Ab)=\int_\Omega \left(|(\nabla-ih\Ab)\psi|^2-\kappa^2|\psi|^2+\frac{\kappa^2}2|\psi|^4\right)\,dx+
h^2\int_\Omega|\curl(\Ab-\Fb_{\rm AB})|^2\,dx\,.
\end{equation}
Note that $\Ea(h)>-\infty$ because
\begin{equation}\label{eq:GL-AB*}
\begin{aligned}
\Eab(\psi, \Ab)&=\int_\Omega \left(|(\nabla-ih\Ab)\psi|^2+\frac{\kappa^2}2(1-|\psi|^2)^2+h^2|\curl(\Ab-\Fb_{\rm AB})|^2\right)\,dx-\frac{\kappa^2}2|\Omega|\\
&\geq -\frac{\kappa^2}2|\Omega|\,.
\end{aligned}
\end{equation}
Note that the energy in \eqref{eq:enAB} depends on $\kappa$, so we will denote it by $\Ea(h;\kappa)$ when we would like to stress its dependence on $\kappa$.

In the Physics literature, it is more appropriate to integrate the magnetic energy on the whole plane $\R^2$, i.e. to  minimize the following energy functional
\[ \Eab^{\R^2}(\psi, \Ab)=\int_\Omega \left(|(\nabla-ih\Ab)\psi|^2-\kappa^2|\psi|^2+\frac{\kappa^2}2|\psi|^4\right)\,dx+
h^2\int_{\R^2}|\curl(\Ab-\Fb_{\rm AB})|^2\,dx\]
on the space $H^1_{h\Fb_{\rm AB}}(\Omega;\C)\times (H^1(\R^2;\R^2)+\Fb_{\rm AB})$. This will yield the same ground state energy as in \eqref{eq:enAB}, by the simple connectivity of the domain $\Omega$ (see \cite[Sec.~10.5, p.~154]{FH-b}).
\Bk
\subsection{Magnetic steps}

For all $\varepsilon\in(0,\varepsilon_0)$, define the vector field
\begin{equation}\label{eq:mstep}
\Fb_\varepsilon(x)=
\begin{cases}
\Fb_{\rm AB}(x)&~{\rm if~}|x|>\varepsilon\,,\medskip\\
\displaystyle\frac1{\pi\varepsilon^2}\Ab_0(x)&~{\rm if~}|x|<\varepsilon\,,
\end{cases}
\end{equation}
where $\Ab_0(x):=\frac12(-x_2,x_1)$. Note that $\Fb_\varepsilon\in H^1_{\rm loc}(\R^2;\Bbb R^2)$ and generates the following magnetic field
\begin{equation}\label{eq:msteps*}
B_\varepsilon:=\curl\Fb_\varepsilon=\frac1{\pi\varepsilon^2}\mathbf 1_{D(0,\varepsilon)}\,,
\end{equation}
which is an example of a \emph{magnetic step}.

One interesting feature of magnetic steps is their manifestation of quantum mechanical  \emph{edge states}, a celebrated phenomenon extensively studied for \emph{linear} models   \cite{RP,HPRS,DHS}. For superconductors with large Ginzburg-Landau parameter, magnetic steps also  produce {\it edge states} in a non-linear framework  \cite{AKP} and enjoy an interesting analogy with \emph{piece-wise smooth} domains \cite{A, BF, CG, CG1} at the onset of superconductivity.

\subsection{From magnetic steps to Aharonov-Bohm}\label{sec:m-step}

We show a new feature of magnetic steps related to the Aharonov-Bohm potential. The connection can be seen \emph{formally} by comparing \eqref{eq:F-AB*} and
$$B_\varepsilon\to\delta_0\quad{\rm in~}\mathcal D'(\R^2)\q\text{as }\var\to 0\,.$$
To make this formal comparison precise, we introduce the following space
\begin{equation}\label{eq:spaceH}
\mathcal H=H^1(\Omega;\C)\times H^1(\Omega;\R^2)\,,
\end{equation}
and the following ground state energy
\begin{equation}\label{eq:en-e}
\Ee(h)=\inf\{\Eef(\psi,\Ab)~:~(\psi,\Ab)\in \mathcal H\}\,,
\end{equation}
where $\Eef(\psi,\Ab)$ is defined by \eqref{eq:GL} for $\Fb=\Fb_\varepsilon$, i.e.
\begin{equation}\label{eq:GL-e}
\Eef(\psi, \Ab)=\int_\Omega \left(|(\nabla-ih\Ab)\psi|^2-\kappa^2|\psi|^2+\frac{\kappa^2}2|\psi|^4\right)\,dx+
h^2\int_\Omega|\curl(\Ab-\Fb_{\varepsilon})|^2\,dx\,.
\end{equation}
The point now is to compare the ground state energies $\Ea(h)$ and $\Ee(h)$ (see \eqref{eq:GL-e} and \eqref{eq:enAB}).

\begin{thm}\label{thm:ms}~
\begin{enumerate}
\item For all $h>0$,
$$\lim_{\varepsilon\to0_+}\Ee(h)=\Ea(h)\,.$$
\item The function $h\mapsto \Ea(h)$ is $2\pi$-periodic. 
\item If $h\in2\pi\Z$,
$$\Ea(h)=-\frac{\kappa^2}2|\Omega|$$
and the energy $\Eab(\psi, \Ab)$ in \eqref{eq:GL-AB} is minimized for
$$(\psi:=e^{i \frac{h}{2\pi}\theta},\Ab:=\Fab)\,,$$
where $(r,\theta)$ denote the polar coordinates in $\R^2$.
\end{enumerate}
\end{thm}

Theorem \ref{thm:ms} suggests that for small $\var>0$, $E_\var(h)$ is approximately periodic in $h$.

\subsection{Transition to the normal state}

Given $\kappa,h>0$, a critical point $(\psi,\Ab)_{\kappa,h}$ of \eqref{eq:GL-AB}  is said to be a \emph{normal} solution (or {\it trivial} solution) if $\psi=0$ everywhere in $\Omega$; if $\psi$ is not identically $0$ on $\Omega$, the critical point is said to be a \emph{superconducting} solution.

In  \emph{generic} situations, all critical points become \emph{normal} solutions after sufficiently increasing the intensity of the applied magnetic field  \cite{GP, LuP}. For the Aharonov-Bohm potential, we will prove that such  transition does \emph{not} occur. In fact, every critical point $(\psi,\Ab)_{\kappa,h}$ displays an oscillatory behavior by transitioning back and forth from \emph{normal} to \emph{superconducting} solutions. Examples of this sort are rare in the literature and are usually observed in non-simply connected domains (\cite{FP, HK}). On the opposite, generically, one observes a monotone transition from \emph{normal} to \emph{superconducting} solutions \cite{FH-d, FH-b}  on simply-connected domains.

For all $h\geq 0$, we introduce the eigenvalue
\begin{equation}\label{eq:ev-AB}
\lambda_{\rm AB}(h,\Omega)=\inf\left\{\int_\Omega |(\nabla-ih\Fab)u|^2\,dx~: ~u\in H^1_{h\Fab}(\Omega;\Bbb C),\;\;\int_\Omega|u|^2\,dx=1\right\}\,.
\end{equation}
This is the eigenvalue of the Aharonov-Bohm operator, $-(\nabla-ih\Fab)^2$, defined by the Friedrichs theorem starting from the closed quadratic form (see \cite{BM})
$$
H^1_{h\Fab}(\Omega;\Bbb C)\ni u\mapsto \int_\Omega|(\nabla-ih\Fab)u|^2\,dx\,.$$
 Among the several existing self-adjoint extensions of the Aharonov-Bohm operator in  $L^2(\Omega;\Bbb C)$, we work in this paper with the Friedrichs extension \cite{AT, DS}. It has    compact resolvent, hence the eigenvalue $\lambda_{\rm AB}(h,\Omega)$ is in the discrete spectrum (see \cite{L}).\Bk

Notice that
\begin{equation}\label{eq:GL-ns}
\lambda_{\rm AB}(h,\Omega)<\kappa^2\implies {\rm every~\underline{minimizer}~of~}\Eab{~\rm is~a~\underline{superconducting}~solution}\,.
\end{equation}
This follows by using the test configuration $(tu_h,\Fab)$, with $t$ sufficiently small and $u_h$ an eigenfunction of $\lambda_{\rm AB}(h)$, so that
$$\Ea(h)\leq \Eab(tu_h,\Fab)<0\,.$$
The result in  Theorem~\ref{thm:AB-ns} below complements \eqref{eq:GL-ns}. Its statement involves the constant $C_*(\Omega)$ introduced below, whose definition is related to the following space
\begin{equation}\label{eq:H1-div}
 H^1_{n0}(\Omega,\div0)=\{\mathbf u\in H^1(\Omega;\R^2)~:~{\rm div}\,\mathbf u=0~{\rm in~}\Omega,\;\; \nu\cdot\mathbf u=0~{\rm on~}\partial \Omega\}\,,
\end{equation}
where $\nu$ is the unit outward normal vector on $\partial\Omega$.

We introduce  the following constant
\begin{equation}\label{eq:C*}
C_*(\Omega)= \frac{1}{\lambda^D(\Omega)}\left(2+\frac{|\Omega|^{1/2}}{m_*(\Omega)}\right),
\end{equation}
where
\begin{equation}\label{eq:m*}
m_*(\Omega)=\inf_{\ab\in { H^1_{n0}(\O,\div0)}\setminus\{0\}}\frac{\|\curl\ab\|_{L^2(\Omega)}^2}{\|\ab\|_{L^4(\Omega)}^2}\quad{\rm and}\quad \lambda^D(\Omega)=\inf_{u\in H^1_0(\Omega)\setminus\{0\}}\frac{\|\nabla u\|_{L^2(\Omega)}^2}{\|u\|_{L^2(\Omega)}^2}\,.
\end{equation}
In light of the {compact embedding $H^1(\Omega;\Bbb R^2)\hookrightarrow L^4(\Omega;\Bbb R^2)$} and the celebrated curl-div inequality  (see \cite[Prop.~D.2.1]{FH-b}),
$\|\ab\|_{H^1(\Omega)}\leq C\|\curl\ab\|_{L^2(\Omega)}$, holding in ${H^1_{n0}(\O,\div0)}$, we see that\break$m_*(\Omega)>0$.
That $\lambda^D(\Omega)>0$ is a consequence of the Poincar\'e inequality; this is also the principal eigenvalue of the Dirichelt Laplacian on $\Omega$.

\begin{thm}\label{thm:AB-ns}
If $\kappa$ and $h$ satisfy
$$
0<\kappa^2< \big(1+C_*(\Omega)\big)^{-1}\lambda_{\rm AB}(h,\Omega)\,,
$$
then every critical point $(\psi,\Ab)_{\kappa,h}$ of the functional $\Eab$ is a \emph{normal} solution.
\end{thm}

\begin{rem}\label{rem:k<l}
Based on the existing results in the case of a uniform applied magnetic field and large Ginzburg-Landau parameter {\cite{LuP, FH-d}}, one would  expect that the result in Theorem~\ref{thm:AB-ns}   holds asymptotically \Bk for $\kappa^2\leq \lambda_{\rm AB}(h,\Omega)$. However, in our present situation, the result is valid without restriction on the asymptotic behavior of $(\kappa,h)$, and for this reason, the minimizing magnetic potential $\Ab$ is no more close to the applied potential $\Fab$, so that the constant $C_*(\Omega)$ can not be neglected in our estimates. 
It would be desirable to prove the existence of a sharp constant $k_*(\Omega)$, independent of $h$, such that the conclusion in Theorem~\ref{thm:AB-ns} holds for $0<\kappa^2<k_*(\Omega)\lambda_{\rm AB}(h,\Omega)$ and fails otherwise.  \Bk
\end{rem}

\subsubsection*{Discussion of Theorems~\ref{thm:ms}\,\&\,\ref{thm:AB-ns}}\

For fixed $h$ we define the critical value of the Ginzburg-Landau parameter $\k$ by
\eq
\k_{c}(h)=\inf\{\k>0: \text{the global minimizers of $\Eab$ are non-trivial}\}.
\eeq
Since non-trivial minimizers yield a negative ground state energy $\Ea(h;\kappa)$ (see \eqref{eq:enAB}), we may express $\k_c(h)$ alternatively as follows
\[
\k_c(h)=\inf\{\kappa>0:~\Ea(h;\kappa)<0\}
\,.\]
Theorem~\ref{thm:AB-ns} yields that  $\kappa_c(h)>0$. The conclusion (2) of Theorem \ref{thm:ms} says that the function $h\mapsto\k_{c}(h)$ is $2\pi$-periodic. Recall that the value of $\k$ depends on the material of the sample and the environment temperature. In the case where the value of $\k$ depends monotonically on the temperature, the periodicity of $\k_{c}(h)$ suggests that, subjected to the applied magnetic field $\mathrm H=h\delta_0$ which is produced by the potential $h \Fab$, the critical temperature of the superconductor is periodic in $h$. This phenomenon is consistent with the periodicity of critical temperature of  a thin cylindrical superconductor in an axial magnetic field  \cite[Fig. 1]{LP}. The complete verification of the ``periodicity  in the quadratic background'' observed in the Little-Parks experiments \cite{LP} requires further study of the Ginzburg-Landau model with an applied magnetic field
$$\mathrm H=h(\delta_0+\curl\E)
$$
with $\curl\E\neq 0$. Here we mention an interesting progress  made in \cite{FP}.

The behavior of the eigenvalue $\lambda_{\rm AB}(h,\Omega)$, displayed in Theorem~\ref{thm:ev-AB} below, is reminiscent of the one in a domain with a single hole \cite{HOOO}.

\begin{thm}\label{thm:ev-AB}
The function $h\mapsto \lambda_{\rm AB}(h,\Omega)$ is continuous, $2\pi$-periodic and satisfies
$$
0=\lambda_{\rm AB}(0,\Omega)<\lambda_{\rm AB}(h)\leq \lambda_{\rm AB}(\pi,\Omega),\q {\forall\,h\in(0,2\pi)}\,.
$$
\end{thm}

In the special case where $\Omega$ is the disc $D(0,R)$, Theorem~\ref{thm:ev-AB} follows easily by decomposition into Fourier modes and separation of variables \cite[Prop.~2.1]{KP}.

Combining the results in \eqref{eq:GL-ns}, Theorems~\ref{thm:AB-ns} and \ref{thm:ev-AB}, we can display the oscillations in the critical points of the functional $\Eab$.

\begin{corol}\label{corol:AB}
Assume that $\kappa$ satisfies
$$
0<\kappa^2<(1+C_*(\Omega))^{-1}\lambda_{\rm AB}(\pi,\Omega).
$$
Consider the sequence $(h_n= \pi n)_{n\geq 1}$. It holds the following.
\begin{enumerate}
\item If $n$ is even, then every minimizer $(\psi,\Ab)_{\kappa, h_n}$ of $\Eab$ is a \emph{superconducting} solution.\medskip
\item If $n$ is odd, then every critical point $(\psi,\Ab)_{\kappa, h_n}$ of $\Eab$ is a \emph{normal} solution.
\end{enumerate}
\end{corol}

In the \emph{disc}  case, we recover the result obtained in our previous work \cite{KP}. However the result  in \cite{KP} is valid {on a disc} under the weaker condition when $\kappa^2<c_*^2\lambda_{\rm AB}(\pi,D(0,R))$, where $c_*$ is a  constant satisfying $c_*^2< 1+C_*(D(0,R))$, see \eqref{eq:C*}.

 One more application of Theorem~\ref{thm:ev-AB} concerns the stability of  the normal solution, i.e. whether it is a local minimizer of the functional $\Eab$ (so far we were concerned whether  it is a global minimizer). In Corollary~\ref{cor:ev-AB-s-n} below, we see  oscillations of the stability of the normal solution as $h$ varies.

\begin{corol}\label{cor:ev-AB-s-n}
We introduce the following critical Ginzburg-Landau parameter
$$\k_c^{\rm loc}(h)=\sqrt{\lambda_{\rm AB}(h;\Omega)}\,.$$ 
Then, the following properties hold.
\begin{enumerate}
\item For $\kappa\leq \k_c^{\rm loc}(h)$, the normal solution is a local minimizer of $\Eab$ \,.
\item For $\kappa>\k_c^{\rm loc}(h)$, the normal solution is not a local minimizer of $\Eab(h)$, and every global minimizer is non-trivial\,.
\item The function $h\mapsto\k_c^{\rm loc}(h)$ is $2\pi$-periodic.
\end{enumerate}
\end{corol}
The hessian of the functional $\Eab$ near the normal solution $(0,\Fab)$ is given by the quadratic form
\begin{equation}\label{eq:loc-min}
(\varphi,\ab)\mapsto \frac{d^2}{d\epsilon^2}\Eab(\epsilon\varphi,\Fab+\epsilon\ab)\Big|_{\epsilon=0}= 2\int_\Omega \Big(|(\nabla-ih\Fab)\varphi|^2-\kappa^2|\varphi|^2+h^2|\curl\ab|^2\Big)\,dx\,.
\end{equation}

Item~1 in Corollary~\ref{corol:AB} then follows since the quadratic form in \eqref{eq:loc-min} will be positive. Item~2 follows by \eqref{eq:GL-ns}. Item~3 follows from Theorem~\ref{thm:ev-AB}.
\Bk 

\subsection{Lack of strong diamagnetism for the magnetic step model}

 We return back to the magnetic step model introduced in Sec.~\ref{sec:m-step}, which, by Theorem~\ref{thm:ms}, converges to the Aharonov-Bohm model. It is then natural to see normal-superconducting oscillations in the same manner observed in Corollary~\ref{corol:AB}.   

We introduce the eigenvalue $\lambda(h\Fb_\varepsilon,\Omega)$ as follows
\begin{equation}\label{eq:ev-ms}
\lambda(h\Fb_\varepsilon,\Omega)=\inf\{ \|(\nabla-ih\Fb_\varepsilon)u\|^2_{L^2(\Omega)}~:~u\in H^1(\Omega;\C)~\&~\|u\|_{L^2(\Omega)}=1\}\,,
\end{equation} 
where $\Fb_\varepsilon$ is the magnetic potential in \eqref{eq:mstep}.

The next theorem is the magnetic step analogue of the results in Theorem~\ref{thm:ev-AB} and Corollary~\ref{corol:AB}, and is in fact consistent with the conclusion of Theorem~\ref{thm:ms}.

\begin{thm}\label{thm:osi-ms}
There exists $\varepsilon_0>0$ such that, for all $\varepsilon\in(0,\varepsilon_0)$, the following holds.
\begin{enumerate}
\item The function $h\mapsto \lambda(h\Fb_\varepsilon,\Omega)$ is not monotone increasing.
\item Assume that 
$
0<\kappa^2<(1+C_*(\Omega))^{-1}\lambda_{\rm AB}(\pi,\Omega)$, where $C_*(\Omega)$ and $\lambda_{\rm AB}(\pi,\Omega)$ are introduced in \eqref{eq:C*} and \eqref{eq:ev-AB} respectively. Let $h_n= \pi n$ with $n$ a positive integer. Then, for $\varepsilon$ sufficiently small:
\begin{enumerate}
\item If $n$ is even,  every minimizer $(\psi_\varepsilon,\Ab_\varepsilon)_{\kappa, h_n}$ of $\Eef$ is a \emph{superconducting} solution.
\item If $n$ is odd,  every critical point $(\psi_\varepsilon,\Ab_\varepsilon)_{\kappa, h_n}$ of $\Eef$ is a \emph{normal} solution.
\end{enumerate}
\end{enumerate}
\end{thm}

The first part in Theorem~\ref{thm:osi-ms} is a new counterexample of strong diamagnetism in a general simply connected domain. The second part displays oscillations in the Little-Parks framework for general simply connected domains too. Such phenomena where observed in disc domains \cite{FP, KS} or in tubular non-simply connected domains \cite{HK}.  \Bk

\section{Magnetic Sobolev space}\label{sec:vs}

\subsection{Hypotheses}\label{sec:hyp}

Throughout this section, we assume that
\begin{itemize}
\item $U\subset \R^2$ is   open\Bk, bounded, simply connected and with a {$C^2$} boundary\,;
\item $I=\{x_1,\cdots,x_N\}\subset U$\,;
\item $U_N=U\setminus\{x_1,\cdots,x_N\}$\,;
\item  {$\fb\in L^q(U;\Bbb R^2)\cap L^2_{\rm loc}(U_N;\Bbb R^2)$ is a given vector field}, with $1<q<2$\,;
\item $\exists\,\varepsilon_0\in(0,1)$, $\forall\,i\in\{1,\cdots,N\}$, $D(x_i,\varepsilon_0)\subset U$\,.
\end{itemize}
For all $\varepsilon\in(0,\varepsilon_0)$, we set $I_\varepsilon=\bigcup\limits_{i=1}^N \overline{D(x_i,\varepsilon)}$\,; clearly $I_\varepsilon\subset U$.

\subsection{Definition of the magnetic Sobolev space}

If $\psi\in L^2(U;\C)$ and $\fb\in L^2_{\rm loc}(U;\R^2)$, then {$\fb\psi\in L^1_{\rm loc}(U;\Bbb C^2)$} and can be viewed as a distribution, i.e { $\fb\psi\in\mathcal D'(U;\Bbb C^2)$.} This allows us to define the magnetic Sobolev space { $H^1_{\mathbf f}(U;\Bbb C)$} in a straightforward manner. For a function  { $\psi\in L^2(U;\Bbb C)$} to be in { $H^1_{\fb}(U;\Bbb C)$,} we simply demand that the distribution
$\nabla\psi-i\fb\psi$ is  a function { belonging to $L^2(U;\Bbb C^2)$.}

As long as { $\fb\not\in L^2_{\rm loc}(U;\Bbb R^2)$,} we can not insure any more that { $\fb\psi\in \mathcal D'(U;\Bbb C^2)$,} and the condition $\nabla\psi-i\fb\psi\in { L^2(U;\Bbb C^2)}$ will be meaningless, since we can not assign a distributional sense of $\nabla\psi-i\fb\psi$ in the whole domain $U$.

However, working under the hypotheses in Sec.~\ref{sec:hyp},  we can define the  magnetic Sobolev space
{ $H^1_{\fb}(U_N;\Bbb C)$}
since we know that { $\fb\in L^2_{\rm loc}(U_N;\Bbb R^2)$}, where $U_N=U\setminus \{x_1,\cdots,x_N\}$. Note that we do not distinguish between the spaces  $L^2(U)$ and $L^2(U_N)$, because the set $I=\{x_1,\cdots,x_N\}$ is finite.

Let us examine more closely  this particular situation.
Pick { $\psi\in L^2(U;\Bbb C)$; since $\fb\in L^2_{\rm loc}(U_N;\Bbb R^2)$,} we can view $\fb\psi$ as a distribution on $U_N$, i.e. { $\fb\psi\in \mathcal D'(U_N;\Bbb C^2)$;  the condition $(\nabla-i\fb)\psi\in L^2(U;\Bbb C^2)$} then means

 \begin{equation}\label{eq:cond-mag-der}
 {\exists\,\mathbf g\in L^2(U;\Bbb C^2),~\forall\,\varphi\in C_c^\infty(U_N;\Bbb C),~\int_{U} \psi (\p_j+i\fb_j)\varphi\,dx=-\int_{U}g_j\varphi\,dx\,,\q j=1,2.
 }
\end{equation}
The problem is that we can not utilise a test function { $\varphi\in C_c^\infty(U;\Bbb C)$,} since the integral of $\fb\psi\varphi$ on $U$ would not make sense.

That is the motivation for the following general definition of the magnetic Sobolev space.

\begin{defn}\label{def:H1-mag}
Under the Hypotheses of  Sec.~\ref{sec:hyp}, we define the corresponding magnetic Sobolev space on $U$ as follows
$$
{ H^1_{\fb}(U;\Bbb C)=\{\psi\in L^2(U;\Bbb C)}~:~\eqref{eq:cond-mag-der}~{\rm holds}\}\,.$$
\end{defn}

\begin{rem}\label{rem:mag-H1}  The characterization of Sobolev spaces by means of the notion of absolute continuity on lines  yields the pleasant property that
$W^{1,q}(U)=W^{1,q}(U_N)$ for all $q\in[1,+\infty)$ (see \cite[Thm.~6.1.3]{Hetal}). Consequently, if {$\fb\in L^2_{\rm loc}(U;\Bbb R^2)$,} then Definition~\ref{def:H1-mag} coincides with the usual one, i.e. the following condition holds for { $\psi\in H^1_\fb(U;\Bbb C)$} (compare with \eqref{eq:cond-mag-der}):
\begin{equation}\label{eq:cond-mag-der*}
{
\exists\,\mathbf g\in L^2(U;\Bbb C^2),~\forall\,\varphi\in C_c^\infty(U;\Bbb C),~\int_{U} \psi (\p_j+i\fb_j)\varphi\,dx=-\int_{U}g_j\varphi\,dx\,,\q j=1,2\;.
}
\end{equation}
Indeed, supposing  that { $\psi\in L^2(U;\Bbb C)$ such that  $\mathbf g=(\nabla-i\fb)\psi$ in $\mathcal D'(U_N;\Bbb C^2)$,  we see that $\psi\in W^{1,1}(U_N\cap K;\Bbb C)= W^{1,1}(U\cap K;\Bbb C)$} for any compact set $K\subset U$, and $\nabla\psi=\mathbf g+i\fb\psi$ becomes a locally integrable function on $U$, so $(\nabla-i\fb)\psi=\mathbf g$ in { $\mathcal D'(U;\Bbb C^2)$} too.
\end{rem}

Fortunately, our hypotheses on $\fb$ will allow us to view $\fb\psi$ as a distribution on $U$ whenever { $\psi\in H^1_{\fb}(U;\Bbb C)$,} thereby  overcoming the technical difficulties in Definition~\ref{def:H1-mag}. This is due to the following lemma.

\begin{lem}\label{lem:H1-mag}
Under the hypotheses in Sec.~\ref{sec:hyp},  $H^1_{\fb}(U;\Bbb C)\subset L^p(U;\Bbb C)$, for all $p\in[2,+\infty)$.
Furthermore, for all $\psi\in H^1_{\fb}(U;\Bbb C)$,  we have
$$\fb\psi\in L^1(U,\Bbb C^2),\quad
|\psi|\in H^1(U)\quad{\rm  and}\quad \|\nabla|\psi|\|_{L^2(U)}\leq \|(\nabla-i\fb)\psi\|_{L^2(U)}\,.$$
\end{lem}

\begin{rem}\label{rem:H1-mag}
In light of Lemma~\ref{lem:H1-mag}, we see that { $\fb\psi\in\mathcal D'(U;\Bbb C^2)$,} and now, we can interpret the condition { $(\nabla-i\fb)\psi\in L^2(U;\Bbb C^2)$}  as follows
$$\exists\,g\in L^2(U),\quad (\nabla-i\fb)\psi=g~{\rm in~}\mathcal D'(U)\,.$$
We then can define the magnetic Sobolev space $H^1_{\fb}(U)$ as follows
\begin{equation}\label{eq:H1-mag-sing}
{
H^1_{\fb}(U;\Bbb C)=\{\psi\in L^2(U;\Bbb C)}~:~\eqref{eq:cond-mag-der*}~{\rm holds}\}\,.
\end{equation}
\end{rem}
%

\begin{proof}[Proof of Lemma~\ref{lem:H1-mag}]~

Let { $\psi\in H^1_{\fb}(U;\Bbb C)$} and consider the distributional derivative, $\mathbf u:=\nabla|\psi|$, in $\mathcal D'(U_N)$. We first  check  that $\mathbf u$ is a measurable vector function. Indeed, we will prove that, in $\mathcal D'(U_N)$,
\begin{equation}\label{eq:nabla|psi|}
{ \mathbf u= \mathbf 1_{\{\psi(x)\not=0\}}\Re \frac{\overline{\psi(x)}}{|\psi(x)|}\,\nabla\psi(x)\,.}
\end{equation}
Pick an arbitrary test function $\varphi\in C_c^\infty(U_N)$.
Assume that $I=\{x_1,\cdots,x_N\}$, $\varepsilon_0>0$ and  ${\rm supp}\,\varphi\subset U\setminus I_{\varepsilon_0}$ (see Sec.~\ref{sec:hyp}). Since { $\fb\in L^2_{\rm loc}(U\setminus I;\Bbb R^2)$,} we know that  { $\psi\in H^1(U\setminus I_{\varepsilon_0};\Bbb C)$} and \eqref{eq:nabla|psi|} holds in $\mathcal D'(U\setminus I_{\varepsilon_0})$, hence
$$
{
\int_{U} |\psi(x)|\,\nabla\varphi (x)\,dx
=-\Re \int_U  \mathbf 1_{\{\psi(x)\not=0\}}\frac{\overline{\psi(x)}}{|\psi(x)|}\,\nabla\psi(x)\,\varphi(x)\,dx\,.
}
$$
Next, we check that the function $\nabla|\psi|$ is in $L^2(U)$. Let $\varepsilon\in(0,\varepsilon_0)$.
By the diamagnetic inequality \cite[Thm.~7.21, pp. 193]{LL}, for  almost every $x\in  U\setminus I_\varepsilon$,
$|\,\nabla|\psi|\,|\leq |(\nabla-i\fb)\psi|$. Consequently, $$\int_{ U_\varepsilon}
|\,\nabla|\psi|\,|^2\,dx\leq \|(\nabla-i\fb)\psi\|_{L^2(U)}^2\,,$$
and by monotone convergence,
$$\|\nabla|\psi|\|_{L^2(U\setminus I)}^2=\lim_{\varepsilon\to0_+}\int_{ U_\varepsilon}
|\,\nabla|\psi|\,|^2\,dx\leq \|(\nabla-i\fb)\psi\|^2_{L^2(U)}<+\infty\,.$$
This proves that $|\psi|\in H^1(U\setminus I)$ (which coincides with the space  $H^1(U)$). Recalling the Sobolev embedding, $H^1(U) \hookrightarrow L^p(U)$ for $p\in[2,+\infty)$, we finish the proof of Lemma~\ref{lem:H1-mag}.
\end{proof}

A useful variant of Lemma~\ref{lem:H1-mag} is given below.

\begin{lem}\label{prop:vs}
For all $(\psi,\ab)\in { H^1_{\fb}(U;\Bbb C)}\times \big(H^1(U;\R^2)+\fb\big)$, it  holds,
\begin{enumerate}
\item $(\nabla-i\ab)\psi\in { L^2(U;\Bbb C^2)}$\,;
\item $\|\nabla|\psi|\|_{L^2(U)}\leq \|(\nabla-i\ab)\psi\|_{L^2(U)}$.
\end{enumerate}
\end{lem}
\begin{proof}
Let $\af=\ab-\fb$. We know that $\af\in H^1(U;\R^2)\hookrightarrow L^4(U;\R^2)$. Consequently, { $\af\, \psi\in L^2(U;\Bbb C)$,} by H\"older's inequality and  Lemma~\ref{lem:H1-mag}.\Bk Thus
$$\|(\nabla-i\ab) \psi\|_{L^2(U)}\leq \|(\nabla-i\fb)\psi\|_{L^2(U)}+\|\af\psi\|_{L^2(U)}<+\infty\,.$$
This proves (1). Noting that $\af\in { L^2(U;\Bbb R^2)}$, we see that (2) follows from Lemma~\ref{lem:H1-mag}.
\end{proof}

\subsection{Compactness in the magnetic Sobolev space}

In the next sections, we will work with minimizing sequences of the functional with Aharonov-Bohm potential. We describe here the procedure of extracting convergent sub-sequences.

We continue to work with hypotheses in Sec.~\ref{sec:hyp}  and under  {\bf the additional assumption}  { $\fb\in L^\infty_{\rm loc}(U\setminus I;\Bbb R^2)$.} Recall the space ${ H^1_{n0}(U,\div0)}$, of divergence free vector fields, introduced earlier in \eqref{eq:H1-div}.

\begin{prop}\label{prop:min-seq}
Let $M>0$. Assume that  $(\psi_n,\af_n)_{n\geq 1}\subset H^1_{\fb}(U;\C)\times { H^1_{n0}(U,\div0)}$ such that:
$$\forall\,n\geq 1\,,~\|(\nabla-i(\af_n+\fb))\psi\|_{L^2(U)}+\|\psi\|_{L^4(U)}+\|\curl\af_n\|_{L^2(U)}\leq M\,.$$
The following { holds}
\begin{enumerate}
\item The sequences $(|\psi_n|)_{n\geq 1}$ and $(\af_n)_{n\geq 1}$ are bounded in $H^1(U)$ { and in $H^1(U,\Bbb R^2)$ respectively}\,;
\item The sequence $((\nabla-i\fb)\psi_n)_{n\geq 1}$ is bounded in { $L^2(U;\Bbb C^2)$\,;}
\item For all $\varepsilon\in(0,\varepsilon_0)$,  the sequence $(\psi_n)_{n\geq 1}$  is bounded in { $H^1(U\setminus I_\varepsilon;\Bbb C)$\,;}
\item There exist $(\psi,\af)\in H^1_{\fb}(U;\C)\times H^1_{\rm div} (U;\R^2)$ and a subsequence $(\psi_{n_k},\af_{n_k})_{k\geq 1}$ such that
$$\begin{aligned}
&\liminf_{k\to+\infty}\|(\nabla-i(\af_{n_k}+\fb))\psi_{n_k}\|_{L^2(U)}\geq \|(\nabla-i(\af+\fb))\psi\|_{L^2(U)}\\
&\lim_{k\to+\infty}\|\psi_{n_k}\|_{L^p(U)}=\|\psi\|_{L^p(U)}\quad(p\in\{2,4\})\\
&\liminf_{k\to+\infty}\|\curl\af_{n_k}\|_{L^2(U)}\geq \|\curl\af\|_{L^2(U)}\,.
\end{aligned}$$
\end{enumerate}
\end{prop}
\begin{proof}~

\begin{center}{\bf Step~1. \it Proof of (1)-(3).}\end{center}

By Lemma~\ref{prop:vs}, the sequence $(|\psi_n|)_{n\geq 1}$ is bounded in $H^1(U)$. By the curl-div inequality \cite[Prop.~D.2.1]{FH-b}, there exists $C>0$ such that,
$$
\forall\,\mathbf u\in { H^1_{n0}(U,\div0)}\,,\quad \|\mathbf u\|_{H^1(U)}\leq C\|\curl\mathbf u\|_{L^2(U)}\,.$$
This proves (1).  By the Sobolev embedding { $H^1(U;\Bbb R^2)\hookrightarrow L^p(U;\Bbb R^2)$}, we get that $(\af_n)_{n\geq 1}$ is bounded in { $L^p(U;\Bbb R^2)$}, for all $p\in[2,\infty)$. By H\"older's inequality,
$$
\|\af_n\psi_n\|_{L^2(U)}\leq \|\af_n\|_{L^4(U)}\|\psi_n\|_{L^4(U)}\,,
$$
and we get that $(\af_n\psi_n)_{n\geq 1}$ is bounded in { $L^2(U;\Bbb C^2)$.} By the Minkowski inequality,
$$
\|(\nabla-i\fb)\psi_n\|_{L^2(U)}\leq \|(\nabla-i(\af_n+\fb))\psi_n\|_{L^2(U)}+\|\af_n\psi_n\|_{L^2(U)}\,,
$$
which proves (2). Since $U$ is bounded, { $L^4(U;\Bbb C)\hookrightarrow L^2(U;\Bbb C)$}, hence $(\psi_n)_{n\geq 1}$ is bounded in { $L^2(U;\Bbb C)$}. Furthermore,  { $\fb\in L^\infty(U\setminus I_\varepsilon;\Bbb R^2)$} and
$$
\|\nabla\psi_n\|_{L^2(U\setminus I_\varepsilon)}\leq
\|(\nabla-i\fb)\psi_n\|_{L^2(U)}+\|\fb\psi_n\|_{L^2(U\setminus I_\varepsilon)},
$$
which proves (3).\medskip

\begin{center}{\bf Step~2. \it  Extraction of the subsequence}\end{center}

By a diagonal sequence argument, the Banach-Alaoglu theorem and the compactness of the embedding $H^1(U)\hookrightarrow L^p(U)$,  $p\in[2,\infty)$, we can extract a subsequence $(\psi_{n_k},\af_{n_k})_{k\geq 1}$, functions $\psi\in H^1_{\rm loc}(U\setminus I;\C)$, { $\zeta\in H^1(U)$,  $\af\in H^1_{n0}(U,\div0)$ and $w\in L^2(U;\Bbb C^2)$}   such that
\begin{align*}
&\psi_{n_k}\rightharpoonup \psi~{\rm in~}H^1_{\rm loc}(U\setminus I;\C)\\
&\psi_{n_k}\to \psi~{\rm in~}L^p_{\rm loc}(U\setminus I;\C)\quad(p\in[2,\infty))\\
&|\psi_{n_k}|\to {\zeta}~{\rm in~}L^p(U)\quad(p\in[2,\infty))\\
&\af_{n_k}\rightharpoonup\af~{\rm in~} { H^1_{n0}(U,\div0)}\\
&\af_{n_k}\to \af~{\rm in~}L^p (U  ;\R^2)\quad(p\in[2,\infty))\\
&(\nabla-i\fb)\psi_{n_k}\rightharpoonup w~{\rm in~} { L^2(U;\C^2)}\,.
\end{align*}
\medskip
\begin{center}{\bf Step~3. \it $\psi\in { L^p(U;\Bbb C)}$.}\end{center}

For all $\varepsilon\in(0,\varepsilon_0)$ and $p\in[2,\infty)$,
$ |\psi_{n_k}|\to |\psi|$ in $L^p(U\setminus I_\varepsilon)$, hence $|\psi|={\zeta}$ in $L^p(U\setminus I_\varepsilon)$. By monotone convergence,
$$
0=\lim_{\varepsilon\to0_+}\int_{U\setminus I_\varepsilon}
 \big|\,|\psi|-{\zeta}\,\big|^p\,dx=\int_U   \big|\,|\psi|-{\zeta}\,\big|^p\,dx\,,
$$
hence $|\psi|={\zeta}$ a.e. in $U$. Since ${\zeta}\in H^1(U)\hookrightarrow L^p(U)$ with $p\in[2,\infty)$, we deduce that { $\psi\in L^p(U;\Bbb C)$} and consequently
$$
\lim_{k\to+\infty}\int_U|\psi_{n_k}|^p\,dx=\int_U|\psi|^p\,dx\quad(p\in[2,\infty))\,.
$$
\medskip

\begin{center}{\bf Step~4. \it Convergence in { $L^p(U;\Bbb C)$}.} \end{center}

We  prove that $\psi_{n_k}\to\psi$ in { $L^p(U)$} as follows. Fix $\varepsilon\in(0,\varepsilon_0)$. By the H\"older and Minkowski  inequalities,
$$
\int_{I_\varepsilon}|\psi_{n_k}-\psi|^p\,dx\leq |I_\varepsilon|^{1/2}\left(\int_U |\psi_{n_k}-\psi|^{2p}\,dx\right)^{1/2}\leq \tilde M |I_\varepsilon|^{1/2},
$$
where $\tilde M=\Big(\sup\limits_{k\geq 1}\|\psi_{n_k}\|_{L^{2p}(U)}+\|\psi\|_{L^{2p}(U)} \Big)^{p}<+\infty$; moreover, $\psi_{n_k}\to\psi$ in { $L^p(U\setminus I_\varepsilon;\Bbb C)$}. With this in hand, we deduce that
$$
0\leq \limsup_{k\to+\infty}\int_U|\psi_{n_k}-\psi|^p\,dx=\limsup_{k\to+\infty}\left(\int_{I_\varepsilon}|\psi_{n_k}-\psi|^p\,dx+\int_{U\setminus I_\varepsilon}|\psi_{n_k}-\psi|^p\,dx\right) \leq \tilde M|I_\varepsilon|^{1/2}\,.
$$
Sending $\varepsilon$ to $0$, we get the desired convergence, $\lim\limits_{k\to+\infty}\int_U|\psi_{n_k}-\psi|^p\,dx=0$.\medskip

\begin{center}{\bf Step~5. \it $\psi\in { H^1_{\fb}(U;\Bbb C)}$.}\end{center}

Since $\fb\in { L^q(U;\Bbb R^2)}$ and $\psi\in { L^p(U;\Bbb C)}$ for all $q\in[1,2)$ and $p\in[2,+\infty)$, we get that $\psi$ and $\fb\psi$ are distributions { on $U$}. Hence $(\nabla-i\fb)\psi\in { \mathcal D'(U;\Bbb C^2)}$.

By Step~4 above, we get that
$(\nabla-i\fb)\psi_{n_k}\to(\nabla-i\fb)\psi$ in ${ \mathcal D'(U;\Bbb C^2)}$. In light of Step~2 above, the weak convergence of $(\nabla-i\fb)\psi_{n_k}$ to $w$ in { $L^2(U;\Bbb C^2)$}  yields the convergence in { $\mathcal D'(U;\Bbb C^2)$}, hence the identity
$(\nabla-i\fb)\psi =w$ in { $\mathcal D'(U;\Bbb C^2)$}.
This proves that $(\nabla-i\fb)\psi\in { L^2(U;\Bbb C^2)}$, and since $\psi\in { L^2(U;\Bbb C)}$, we eventually get that $\psi\in { H^1_{\fb}(U;\Bbb C)}$. \medskip

\begin{center}{\bf Step~6. \it End of the proof of (4).}\end{center}

For all $\varepsilon\in(0,\varepsilon_0)$,
$$\int_{U} |(\nabla-i(\af_{n_k}+\fb))\psi_{n_k}|^2\,dx\geq \int_{U\setminus I_\varepsilon}|(\nabla-i(\af_{n_k}+\fb))\psi_{n_k}|^2\,,$$
and $(\nabla-i(\af_{n_k}+\fb))\psi_{n_k}\rightharpoonup (\nabla-i(\af+\fb))\psi$ in { $L^2(U\setminus I_\varepsilon;\Bbb C^2)$} by Step~2 above; this yields
$$\liminf_{k\to+\infty}\int_{U} |(\nabla-i(\af_{n_k}+\fb))\psi_{n_k}|^2\,dx\geq \int_{U\setminus I_\varepsilon}|(\nabla-i(\af +\fb))\psi|^2\,dx\,.$$
By monotone convergence
$$\liminf_{k\to+\infty}\int_{U} |(\nabla-i(\af_{n_k}+\fb))\psi_{n_k}|^2\,dx\geq \lim_{\varepsilon\to0_+}\int_{U\setminus I_\varepsilon}|(\nabla-i(\af +\fb))\psi|^2\,dx=\int_{U}|(\nabla-i(\af +\fb))\psi|^2\,dx\,.$$
Finally, $\curl\af_{n_k}\rightharpoonup \curl\af$ in $L^2(U)$, which yields that
$$\liminf_{k\to+\infty}\int_U|\curl\af_{n_k}|^2\,dx\geq \int_U|\curl\af|^2\,dx\,.$$
\end{proof}

\section{Minimizers  with Aharonov-Bohm potential}\label{sec:minAB}

In this section, we study the existence of minimizers of the GL functional in \eqref{eq:GL-AB} along with some of their properties.

\subsection{Gauge invariance}

Using gauge invariance, we can restrict the minimization of the functional in \eqref{eq:GL-AB} to the space of divergence free vector fields;  the advantage being that such vector fields enjoy pleasant  regularity  properties.

\begin{prop}\label{prop:vsAB}
For all $h>0$,
$$\Ea(h)=\inf\{\Eab(\psi,\mathbf a+\Fb_{\rm AB})~:~(\psi,\mathbf a)\in H^1_{h\Fb_{\rm AB}}(\Omega;\C)\times { H^1_{n0}(\O,\div0)}\}\,,$$
 where
\begin{itemize}
\item $\Ea(h)$ is introduced in \eqref{eq:enAB}\,;
\item $\Eab$ is the functional introduced in \eqref{eq:GL-AB}\,;
\item ${ H^1_{n0}(\O,\div0)}$ is the space  introduced  in \eqref{eq:H1-div}.
\end{itemize}
\end{prop}
\begin{proof}
Let $(\psi,\Ab:=\mathbf a+\Fb_{\rm AB})\in { H^1_{h\Fb_{\rm AB}}(\Omega;\Bbb C)}\times \big(H^1(\Omega;\R^2)+\Fb_{\rm AB}\big)$.  By \cite[Prop.~D.1.1]{FH-b}, there exists $\varphi\in { H^2(\Omega)}$ such that $\tilde{\mathbf a}:=\mathbf a-\nabla\varphi\in { H^1_{n0}(\O,\div0)}$. Setting $\tilde\psi=e^{ih\varphi}\psi$, it is clear that $(\tilde\psi,\tilde{\mathbf a})\in { H^1_{h\Fb_{\rm AB}}(\Omega;\Bbb C)\times H^1_{n0}(\O,\div0)}$ and
$\Eab(\psi,\mathbf a+\Fb_{\rm AB})=\Eab(\tilde\psi,\tilde{\mathbf a}+\Fb_{\rm AB})$.
\end{proof}

\subsection{Existence of minimizers}

Next we establish the existence of minimizing configurations.

\begin{prop}\label{prop:ex-min}
For all $h>0$, there exists $(\psi,\mathbf a)\in H^1_{h\Fb_{\rm AB}}(\Omega;\C)\times { H^1_{n0}(\O,\div0)}$ such that
$$\Eab(\psi,\mathbf a+\Fb_{\rm AB})=\Ea(h)\,.$$
\end{prop}
\begin{proof}
We use the standard method of the calculus of variations. We choose a minimizing sequence $(\psi_n,\mathbf a_n)_{n\geq 1}\subset H^1_{h\Fb_{\rm AB}}(\Omega;\C)\times { H^1_{n0}(\O,\div0)}$ such that
\begin{equation}\label{eq:min-seq}
\lim\limits_{n\to+\infty}\Eab(\psi_n,\mathbf a_n+\Fb_{\rm AB})=\Ea(h)\,.
\end{equation}
By \eqref{eq:GL-AB*}, there exists $M>0$ such that,
$$
\|(\nabla-ih(\ab_n+\Fb_{\rm AB}))\psi_n\|_{L^2(\Omega)}+\|\psi_n\|_{L^4(\Omega)}+ h\|\curl\mathbf a_n\|_{L^2(\Omega)}\leq M\,,\q\forall\,n\geq 1\,.
$$
We can apply Proposition~\ref{prop:min-seq} with $(\psi_n,\af_n=h\ab_n, \fb=h\Fab)$. We get a subsequence $(\psi_{n_k},\ab_{n_k})_{k\geq 1}$ and a configuration $(\psi,\ab)\in { H^1_{h\Fab}(\Omega;\Bbb C)\times H^1_{n0}(\Omega,\div0)}$ such that
$$\Ea(h)=\liminf_{k\to\infty} \Eab(\psi_{n_k},\ab_{n_k}+\Fab)\geq \Eab(\psi,\ab+\Fab)\geq \Ea(h)\,, $$
where the identity on the left hand side follows from \eqref{eq:min-seq}, and the last  inequality  on the right hand side follows from the definition of $\Ea(h)$.
\end{proof}

\begin{defn}\label{def:min-config}
Given $h>0$ and a  $(\psi,\ab)\in { H^1_{h\Fab}(\Omega;\Bbb C)\times  H^1_{n0}(\O,\div0)}$, we will use the following terminology:
\begin{itemize}
\item $(\psi,\ab)_{h}$ is said to be a minimizing configuration of $\Eab$ if $\Eab(\psi,\ab+\Fab)=\Ea(h)$, where $\Eab$ and $\Ea(h)$ are introduced in \eqref{eq:GL-AB} and \eqref{eq:enAB} respectively\,;
\item $(\psi,\ab)_{h}$ is said to be a critical configuration of $\Eab$ if  $\frac{d}{dt}\Eab(\psi+t\varphi,\ab+\Fab)\Big|_{t=0}=0$ and $\frac{d}{dt}\Eab(\psi,\ab+\Fab+t\mathfrak b)\Big|_{t=0}=0$, for all $(\varphi,\mathfrak b)\in { H^1_{h\Fab}(\Omega;\Bbb C)\times  H^1_{n0}(\O,\div0)}$.
\end{itemize}
\end{defn}

\begin{rem}\label{rem:min-config}
Obviously, every minimizing configuration is a critical configuration. Furthermore, every critical configuration $(\psi,\ab)_h$ satisfies
\begin{equation}\label{eq:GLeq}
\left\{\aligned
-&\big(\nabla-ih\Ab\big)^2\psi=\kp^2(1-|\psi|^2)\psi &{\rm in}\ \Omega\,,\\
-&\nabla^\bot  \big(\curl\ab\big)= \frac{1}{h}{\rm Im}\big(\overline{\psi}(\nabla-ih {\bf A})\psi\big) & {\rm in}\ \Omega\,,\\
&\nu\cdot(\nabla-ih {\bf A})\psi=0 & {\rm on}\ \partial \Om \,,\\
&{\rm curl}\,\ab=0 & {\rm on}\ \partial \Om \,,
\endaligned
\right.
\end{equation}
where $\Ab=\ab+\Fab$, $\nu$ is the outward unit normal vector on $\partial\Omega$, and the operator $\nabla^\bot=(-\partial_{x_2},\partial_{x_1})$ is the Hodge gradient.
\end{rem}

\subsection{A priori estimates}

\begin{prop}\label{prop:EL-eq}
There exists $C_0>0$ such that, given a critical configuration $(\psi,\ab)_{h}\in H^1_{h\Fab}(\Omega;\C)\times { H^1_{n0}(\O,\div0)}$ of $\Eab$, the following holds:
\begin{enumerate}
\item $\ab\in H^2(\Omega;\R^2)$\,;\medskip
\item $(\psi,\ab)\in { C^\infty (\overline{\Omega}\setminus\{0\};\Bbb C)\times C^\infty (\overline{\Omega}\setminus\{0\};\Bbb R^2)}$\,;\medskip
\item  $\|(\nabla-ih\Ab)\psi\|_{L^2(\Omega)}\leq \kappa\|\psi\|_{L^2(\Omega)}$ where $\Ab=\ab+\Fab$\,;\medskip
\item
$\|\psi\|_{L^\infty(\Omega)
}\leq 1$\,;\medskip
\item $\|\mathbf a\|_{H^2(\Omega)}\leq \displaystyle\frac{C_0}{h}\kappa \|\psi\|_{L^2(\Omega)} $\,.
\end{enumerate}
\end{prop}
\begin{proof}
Since $\ab\in { H^1_{n0}(\Omega,\div0)}$, the second equation in \eqref{eq:GLeq} yields that $\curl\ab\in H^1(\Omega)$. By the curl-div estimate (see \cite[Prop.~D.2.1]{FH-b}),  { $\ab\in H^2(\Omega;\Bbb R^2)$} and
\begin{equation}\label{eq:curl-div}
\|\ab\|_{H^2(\Omega)}\leq C_\Omega\|\curl\ab\|_{H^1(\Omega)}\,,
\end{equation}
where $C_\Omega<+\infty$ depends on $\Omega$ only. This proves (1).

That $(\psi,\ab)$ is smooth in $\overline{\Omega}\setminus\{0\}$ follows by a bootstrapping argument (see \cite[Prop.~3.6]{SS-b}).

The identity $\frac{d}{dt}\mathcal E(\psi+t\psi,\Ab)\big|_{t=0}=0$ yields that
\begin{equation}\label{eq:E0}
\mathcal E_0(\psi,\Ab):=\int_\Omega
\left(|(\nabla-ih\Ab)\psi|^2-\kappa^2|\psi|^2+\frac{\kappa^2}{2}|\psi|^4\right)\,dx=-\frac{\kappa^2}{2}\int_\Omega|\psi|^4\,dx\leq0\,,\end{equation}
which proves (3).

Now we prove (4). We introduce the following function
$$\tilde\psi=[|\psi|-1]_+\frac{\psi}{|\psi|}:=\begin{cases}
(|\psi|-1)\frac{\psi}{|\psi|}&~{\rm if~}|\psi|\geq 1\\
0&{\rm~ if~}|\psi|\leq 1
\end{cases}\,.$$
By Lemma~\ref{prop:vs}, $|\psi|\in H^1(\Omega)$, hence $\tilde\psi\in H^1_{h\Fb_{\rm AB}}(\Omega;\C)$ because $\psi\in H^1_{\Fb_{\rm AB}}(\Omega;\C)$. So we can use the identity
$\frac{d}{dt}\Eab(\psi+t\tilde\psi,\Ab)=0$, which reads as follows,
$${\rm Re}\int_\Omega\left( (\nabla-ih\Ab)\psi\cdot\overline{(\nabla-ih\Ab)\tilde\psi} +(|\psi|^2-1)\psi\overline{\tilde\psi}\right)\,dx=0\,.$$  The rest of the proof is as
  \cite[Prop.~10.3.1]{FH-b}.

Finally, we prove (5). By the last equation in \eqref{eq:GLeq}, $\curl\ab\in H^1_0(\Omega)$, hence, by the Poincar\'e inequality
$$\|\curl\ab\|_{H^1(\Omega)}\leq C_\Omega'\|\nabla( \curl\ab)\|_{L^2(\Omega)}\,,$$
where $C_\Omega'$ depends on $\Omega$ only. Using \eqref{eq:curl-div} and the second equation in \eqref{eq:GLeq}, we get
$$\|\ab\|_{H^2(\Omega)}\leq \frac{C_\Omega C_\Omega'}{h} \big\|{\rm Im}\big(\overline{\psi}(\nabla-ih {\bf A})\psi\big)\big\|_{L^2(\Omega)}
\leq \frac{C_\Omega C_\Omega'}{h}\kappa\|\psi\|_{L^2(\Omega)}\,,$$
where we used (3) and (4) to write the last inequality.
\end{proof}

\subsection{The non-degenerate case}

Our next result is that for a minimizing configuration the order parameter is actually in the space $H^1(\Omega;\C)$ not just in the magnetic Sobolev space $H^1_{h\Fab}(\Omega;\C)$, except for the degenerate case where $h\in2\pi\Z$. This is related to a \emph{magnetic Hardy} inequality \cite{LW} (see Lemma~\ref{lem:1D-op} below).

\begin{prop}\label{prop:min-AB-H1}
Given $r_0,\kappa, h>0$ such that $D(0,r_0)\subset\Omega$, there exists $C>0$ such that every  minimizing configuration $(\psi,\ab)_h$ of $\Eab$ satisfies:
$$2\pi\alpha(h)\|\Fab\psi\|_{L^2(\Omega)}+\|(\nabla-ih\Fab)\psi\|_{L^2(\Omega)}\leq C\,,$$
where
\begin{equation}\label{eq:alpha(h)}
\alpha(h)=\inf_{n\in\Z}\left|n-\frac{h}{2\pi}\right|\,.
\end{equation}
In particular, for $h\not\in2\pi\mathbb Z$, $\psi\in H^1(\Omega;\C)$.
\end{prop}

The proof relies on the following one dimensional spectral analysis.

\begin{lem}\label{lem:1D-op}
For all $h>0$,
$$
\inf_{u\in H^1_{\rm per}(0,2\pi)\setminus\{0\}}\frac{
\displaystyle\int_{0}^{2\pi}
\left|{ {d\over d\theta} } u-i\frac{h}{2\pi}u\right|^2d\theta}{\displaystyle\int_0^{2\pi}|u|^2d\theta}= \alpha(h)^2\,,$$
where $\alpha(h)$ is introduced in Proposition~\ref{prop:min-AB-H1} and
$$
H^1_{\rm per}(0,2\pi)=\{u\in H^1(0,2\pi)~:~u(0)=u(2\pi)\}.
$$
\end{lem}

\begin{proof}
For all $u\in H^1_{\rm per}(0,2\pi)$ { let $v=e^{-i\frac{h}{2\pi}\theta}u$.} Notice that $v(0)=e^{i h}v(2\pi)$ and
$$
q(u):=\int_{0}^{2\pi}\left|{ {d\over d\theta} } u-i\frac{h}{2\pi}u\right|^2d\theta=\int_{0}^{2\pi}\left|{ {d\over d\theta} } v\right|^2d\theta\,.
$$
So we are led to determine the spectrum of the operator $\mathcal L=-\frac{d^2}{d\theta^2}$ with the boundary condition
\begin{equation}\label{eq:1D-bc}
v(0)=e^{ih}\,v(2\pi)\,.
\end{equation}
It is easy to check that $e^{iw\theta}$ is an eigenfunction associated with the eigenvalue $\lambda:=w^2$, $w\in\R$, and that the boundary condition in \eqref{eq:1D-bc} reads as
$w+\frac{h}{2\pi}\in\Z$\,.  The completeness of the Fourier basis in $L^2(0,2\pi)$ asserts that the spectrum consists exactly of the aforementioned 
eigenvalues. \end{proof}

\begin{proof}[Proof of Proposition~\ref{prop:min-AB-H1}]
Using (4)-(5) in Proposition~\ref{prop:EL-eq}, the H\"older inequality and the Sobolev embedding $H^1(\Omega)\hookrightarrow L^4(\Omega)$, we write
$$\|\ab\,\psi\|_{L^2(\Omega)}
\leq
\|\ab\|_{L^4(\Omega)}\|\psi\|_{L^4(\Omega)}
\leq \frac{\tilde C\kappa}{h}\,,$$
for a constant $\tilde C$ independent of $(\kappa,h)$. Consequently, by the Minkowski inequality and (3) in Proposition~\ref{prop:EL-eq}, we get
$$\|(\nabla-ih\Fab)\psi\|_{L^2(\Omega)}\leq \kappa|\Omega|^{1/2}+\tilde C\kappa\,.$$
Now we express $\|(\nabla-ih\Fab)\psi\|_{L^2(D(0,r_0))}$ in polar coordinates $(r,\theta)$ as follows
$$ \|(\nabla-ih\Fab)\psi\|_{L^2(D(0,r_0))}^2
=\int_0^{2\pi}\int_0^{r_0} \left(|\partial_r\psi|^2+\frac1{r^2}\left|\partial_\theta-i\frac{h}{2\pi}\psi\right|^2\right)rdr d\theta\,.$$
Using Lemma~\ref{lem:1D-op}, we infer the following estimate,
\[
\|(\nabla-ih\Fab)\psi\|_{L^2(D(0,r_0))}^2
\geq \int_0^{2\pi}\int_0^{r_0} \left(\frac{\alpha(h)}{r^2}|\psi |^2\right)rdr d\theta=\big(2\pi\alpha(h)\big)^2\int_{D(0,r_0)}|\Fab\psi|^2\,dx\,.\]
\end{proof}

\begin{rem}\label{rem:ev-R2}
The proof of Proposition~\ref{prop:min-AB-H1} also yields that, for $h\not\in 2\pi\Z$, the space $H^1_{h\Fab}(\Omega;\C)$ is embeded in the space $H^1(\Omega;\C)\cap L^2(\Omega;|\Fab|^2\,dx)$ with the following inequality
\[\|\Fab u\|_{L^2(\Omega)}^2\leq C_h \Big(\|(\nabla-ih\Fab)u\|_{L^2(\Omega)}^2+\|u\|_{L^2(\Omega)}^2\Big)\quad (u\in H^1_{h\Fab}(\Omega;\C))\,,\]
with $C_h$ a constant dependent of $h$.
\end{rem}

\subsection{The degenerate case}

We determine the minimizers of the functional in \eqref{eq:GL-AB} in the degenerate case where $h\in2\pi\Z$.

\begin{prop}\label{prop:deg-case}
Assume that $h=2\pi n_0$ with $n_0\in\Z$. Then,
$$\Ea(h)=-\frac{\kappa^2}2|\Omega|$$
and every minimizer $(\psi,\Ab)$ has the form
$$\psi=c\,e^{in_0\theta}\quad{\rm and}\quad \quad \Ab=\Fab$$
with $c\in\C$ satisfying $|c|=1$.
\end{prop}
\begin{proof}
The inequality $\Ea(h)\geq -\frac{\kappa^2}{2}|\Omega|$ follows from \eqref{eq:GL-AB*}. To obtain the reverse inequality, we write
$$\Ea(h)\leq \Eab(u,\Fab)$$
with  $u=e^{in_0\theta}$.  Using polar coordinates, we notice that
$$|(\nabla-ih\Fab)u|^2=|\partial_r u|^2+\frac1{r^2}\left|\left(\partial_\theta-i\frac{h}{2\pi}\right)u\right|^2=0\,.$$
This proves that $u\in H^1_{h\Fab}(\Omega;\C)$, $\Eab(u;\Fab)=-\frac{\kappa^2}2|\Omega|$ and $(u,\Fab)$ is a minimizer.

Now, assume that $(\psi,\ab)$ is a minimizing configuration of $\Eab$, i.e. $\Eab(\psi,\ab+\Fab)=\Ea(h)=-\frac{\kappa^2}2|\Omega|$. Notice that $\ab=0$, since $\ab\in { H^1_{n0}(\O,\div0)}$ and (see \eqref{eq:GL-AB*})
$$-\frac{\kappa^2}2|\Omega|+h^2\int_\Omega|\curl\ab|^2\,dx\leq \Eab(\psi,\ab+\Fab)=-\frac{\kappa^2}2|\Omega|\,.$$
The same argument yields that $|\psi|=1$ and $(\nabla-ih\Fab)\psi=0$, since
$$ -\frac{\kappa^2}2|\Omega|+\int_\Omega\left(|(\nabla-ih\Fab)\psi|^2+\frac{\kappa^2}2(1-|\psi|^2)^2\right)\,dx \leq \Eab(\psi,\Fab)=-\frac{\kappa^2}2|\Omega|\,.$$
By introducing the ansatz $\psi=c\,u$, we find
$$0=(\nabla-ih\Fab)\psi=c(\nabla-ih\Fab)u+u\nabla c=u\nabla c$$
hence $c$ must be a constant. Finally, the condition $|\psi|=1$ yields that $|c|=1$.

\end{proof}
\section{Minimizers  with  a magnetic step}\label{sec:minMS}

In this section we study the minimizers of the functional $\Eef$ introduced in \eqref{eq:GL-e}. Since this is associated with the magnetic potential $\Fb_\varepsilon\in H^1(\Omega;\R^2)$, the corresponding magnetic field $B_\varepsilon=\curl\Fb_\varepsilon$ is in $L^2(\Omega)$. Hence, we can use the results in
\cite[Thm.~10.2.1]{FH-b}. In particular, for all  $h>0$ and $\varepsilon\in(0,\varepsilon_0]$, there exists a configuration  $(\Psi,\Ab)_{h,\varepsilon}\in\mathcal H$ such that
\begin{equation}\label{eq:min-conf-e}
\Eef(\Psi,\Ab)=\Ee(h)\,,
\end{equation}
where $\Ee(h)$ is introduced in \eqref{eq:en-e}.
A configuration satisfying \eqref{eq:min-conf-e} is said to be a minimizer of $\Eef$.
Similarly as we did in Proposition~\ref{prop:vsAB}, we can use the gauge invariance to select a configuration
\begin{equation}\label{eq:space-h0}
(\psi_\varepsilon,\ab_\varepsilon)_{h}\in \mathcal H_0:=H^1(\Omega;\C)\times { H^1_{n0}(\O,\div0)}
\end{equation}
such that $(\psi_\varepsilon,\Ab_\varepsilon:=\ab_\varepsilon+\Fb_\varepsilon)$ is a minimizer of $\Eef$. Such a configuration is said to be a \emph{minimizing configuration} of $\Eef$. It satisfies the following properties (see \cite[Prop.~10.3.1\,\&\,Lem.~10.3.2]{FH-b}):
\begin{align}
&\|\psi_\varepsilon\|_{L^\infty(\Omega)}\leq 1\,,\label{eq:min-GL-e}\\
&\|(\nabla-ih\Ab_\varepsilon)\psi_\varepsilon\|_{L^2(\Omega)}\leq \kappa\|\psi_\varepsilon\|_{L^2(\Omega)}\,,\label{eq:min-GL-e'}\\
&\|\curl \ab_\varepsilon\|_{L^2(\Omega)}\leq \frac{\kappa^2}{h}\|\psi_\varepsilon\|_{L^2(\Omega)}\,,\label{eq:min-GL-e''}
\end{align}
 Furthermore, $(\psi_\varepsilon,\ab_\varepsilon)_h$ is a solution of
\begin{equation}\label{eq:GLeq-e}
\left\{\aligned
-&\big(\nabla-ih\Ab_\varepsilon\big)^2\psi_\varepsilon=\kp^2(1-|\psi_\varepsilon|^2)\psi_\varepsilon &{\rm in}\ \Omega\,,\\
-&\nabla^\bot  \big(\curl\ab_\varepsilon\big)= \frac{1}{h}{\rm Im}\big(\overline{\psi_\varepsilon}(\nabla-ih {\bf A_\varepsilon})\psi_\varepsilon\big) & {\rm in}\ \Omega\,,\\
&\nu\cdot(\nabla-ih {\bf A}_\varepsilon)\psi_\varepsilon=0 & {\rm on}\ \partial \Om \,,\\
&{\rm curl}\,\ab_\varepsilon=0 & {\rm on}\ \partial \Om \,.
\endaligned
\right.
\end{equation}
Using \eqref{eq:curl-div} and the second equation in \eqref{eq:GLeq-e}, we can prove that (see the proof of Prop.~\ref{prop:EL-eq}-(5)):
\begin{equation}\label{eq:curl-div*}
\|\ab_\varepsilon\|_{H^2(\Omega)}\leq \frac{\tilde C}{h}\,,
\end{equation}
where $\tilde C$ does not depend on $\varepsilon$ and $h$.

In the sequel, we study the behavior of the minimizing configurations of $\Eef$ as $\varepsilon$ approaches~$0$.

\begin{prop}\label{prop:min-MS}
Given $\omega\subset \Omega\setminus\{0\}$ and $h>0$, there exist $\varepsilon_0\in(0,1)$ and $C>0$ such that, for all $\varepsilon\in(0,\varepsilon_0]$,  every minimizing configuration $(\psi_\varepsilon,\ab_\varepsilon)_h$ of $\Eef$ satisfies
$$\|\psi_\varepsilon\|_{H^2(\omega)}\leq C\,.$$
\end{prop}
\begin{proof}
For $\varepsilon_0$ sufficiently small and $\varepsilon\in(0,\varepsilon_0]$,  we have
$$\omega\subset  \Omega \setminus \overline{D(0,2\varepsilon_0)} \subset \Omega\setminus \overline{D(0,\varepsilon)}\,.$$
Hence $\Fb_\varepsilon=\Fab$ on $\tilde\omega:=\Omega\setminus \overline{D(0,\varepsilon_0)}$, it is smooth and the first equation in \eqref{eq:GLeq-e} reads as follows
$$-\Delta\psi_\varepsilon+2ih(\ab_\varepsilon+\Fab)\cdot\nabla\psi_\varepsilon+h^2|\ab_\varepsilon+\Fab|^2\psi_\varepsilon=\kappa^2(1-|\psi_\varepsilon|^2)\psi_\varepsilon~{\rm in~}\tilde\omega\,, $$
since ${\rm div}(\ab_\varepsilon)={\rm div}(\Fab)=0$. Using \eqref{eq:min-GL-e}, \eqref{eq:min-GL-e'} and \eqref{eq:curl-div*}, we get
$$\|\Delta\psi_\varepsilon\|_{L^2(\tilde\omega)}\leq \hat C\,.$$
Note that, on $\partial\Omega$, the following boundary condition holds (which results from the third equation in \eqref{eq:GLeq-e} and the boundary condition on { $\ab_\varepsilon\in H^1_{n0}(\Omega,\div0)$, hence $\nu\cdot\ab_\varepsilon=0$,} see \eqref{eq:H1-div}):
$$\nu\cdot\nabla\psi_\varepsilon=ih(\nu \cdot\Fab)\psi_\varepsilon\,,$$
with $\nu\cdot\Fab$ a continuous function on $
\partial\Omega$, by smoothness of the boundary. Consequently, we can apply the $L^2$-elliptic estimates and finish the proof (see \cite[Thm.~E.4.6]{FH-b}).
\end{proof}

\begin{prop}\label{prop:min-MS*}
Given $\alpha\in(0,1)$, $h>0$ and a sequence $(\varepsilon_n)_{n\geq 1}\subset\R_+$ which converges to $0$, there exist
$(\psi_*,\ab_*)\in H^1_{h\Fab}(\Omega;\C)\times { H^1_{n0}(\O,\div0)}$ and a subsequence
$$(\psi_{\varepsilon_n},\ab_{\varepsilon_n},\varepsilon_n)_{n\in I}\subset H^1(\Omega;\C)\times { H^1_{n0}(\O,\div0)}\times\R_+$$
 such that
\begin{enumerate}
\item $(\psi_{\varepsilon_n},\ab_{\varepsilon_n})_h$ is a minimizing configuration of $\Eefn$\,;\medskip
\item For every open set $\omega\subset\Omega\setminus\{0\}$, $\psi_{\varepsilon_n}\to\psi_*$ in { $H^1(\omega;\Bbb C)$}\,;\medskip
\item $\ab_{\varepsilon_n}\to\ab_*$ in { $H^1(\Omega;\Bbb R^2)$}\,;\medskip
\item $\psi_{\varepsilon_n},\psi_*\in {C^{0,\alpha}_{\rm loc}(\overline{\Omega}\setminus\{0\};\Bbb C)}$ and $\ab_{\varepsilon_n},\ab_*\in {C^{0,\alpha}(\overline{\Omega};\Bbb R^2)}$\,;\medskip
\item  $\psi_{\varepsilon_n}\to\psi_*$ in  { $C^{0,\alpha}_{\rm loc}(\overline{\Omega}\setminus\{0\};\Bbb C)$} and $\ab_{\varepsilon_n}\to\ab_*$  in { $C^{0,\alpha}(\overline{\Omega};\Bbb R^2)$}\,;\medskip
\item $\|\psi_*\|_{L^\infty(\Omega)}\leq 1$\,;\medskip
\item $\psi_{\Be\varepsilon_n}\to\psi_*$ in { $L^p(\Omega;\Bbb C)$}, for all $p\in[2,+\infty)$\,;\medskip
\item $(\nabla-ih\F_{\varepsilon_n})\psi_{\varepsilon_n}\rightharpoonup (\nabla-ih\Fab)\psi_*$ in { $L^2(\Omega;\Bbb R^2)$}\,;\medskip
\item $\liminf\limits_{n\to\infty}\Eefn(\psi_{\varepsilon_n},\ab_{\varepsilon_n}+\Fb_{\varepsilon_n})\geq \Eab(\psi_*,\ab_*+\Fab)$\,,
\end{enumerate}
where $\Eab$ is the functional introduced in \eqref{eq:GL-AB}.
\end{prop}
\begin{proof}
Consider a sequence $(\varepsilon_n)_{n\geq 1}\subset\R_+$ such that $\lim\limits_{n\to\infty}\varepsilon_n=0$. For all $n\geq 1$, choose a minimizing configuration $(\psi_{\varepsilon_n},\ab_{\varepsilon_n})_h$ of $\Eefn$.

By Proposition~\ref{prop:min-MS}, for all $r\in(0,\varepsilon_0]$, there exists $C,N_0>0$ such that
$$
\|\psi_{\varepsilon_n}\|_{H^2(\Omega_r)}\leq C\,,\q \forall\,n\geq N_0\,,
$$
where $\Omega_r=\Omega\setminus \overline{D(0,r)}$. By a diagonal sequence argument, we can construct a function\footnote{Initially, $\psi_*$ is defined on every $\Omega_r$, but can be defined pointwise on all of $\Omega\setminus\{0\}$ as follows. Given $x\in\Omega\setminus\{0\}$, we can choose $r$ so that $x\in\Omega_r$, then we set $\psi_*(x)=\lim\limits_{\varepsilon_n\to0_+}\psi_{\varepsilon_n}(x)$\,.}\break$\psi_*:\Omega\setminus\{0\}\to\C$ and extract a subsequence of
$(\psi_{\varepsilon_n})_{n\in I_0}$ which is  weakly convergent to $\psi_*$ in every { $H^2(\Omega_r;\Bbb C)$}, $r\in(0,\varepsilon_0]$.

At the same time, \eqref{eq:curl-div*} yields a function $\ab_*\in H^2(\Omega;\R^2)$ and a subsequence  $(\ab_{\varepsilon_n})_{n\in I_1\subset I_0}$ which converges weakly to $\ab_*$ in $H^2(\Omega;\R^2)$.

In light of the estimates in \eqref{eq:min-GL-e}-\eqref{eq:min-GL-e''}, we see that the sequence $\big((\nabla-ih\Fb_{\varepsilon_n})\psi_{\varepsilon_n}\big)$ is bounded in { $L^2(\Omega;\Bbb C^2)$}. So we can extract a  subsequence $\big((\nabla-ih\Fb_{\varepsilon_n})\psi_{\varepsilon_n}\big)_{n\in I_2\subset I_1}$ that  is weakly convergent in { $L^2(\Omega;\Bbb C^2)$}, and denote its weak limit by $g$.

By compactness of the embedding $H^2(U)\hookrightarrow H^1(U)$ and $H^2(U) \hookrightarrow C^{0,\alpha}(\overline{U})$, we can extract a further subsequence, $(\psi_{\varepsilon_n},\ab_{\varepsilon_n})_{n\in I\subset I_2}$, which converges to $(\psi_*,\ab_*)$ in { $H^1(\Omega_r;\Bbb C)\times H^1(\Omega;\Bbb R^2)$ and in $C^{0,\alpha}(\Omega_r;\Bbb C)\times C^{0,\alpha}(\Omega;\Bbb R^2)$}.  This proves (2)-(5).

The estimate $\|\psi_*\|_{L^\infty(\Omega)}\leq 1$ follows from \eqref{eq:min-GL-e}; actually, for every $x\in\Omega\setminus\{0\}$, we can find $r>0$ such that $x\in\Omega_r$ and consequently $\psi_{\varepsilon_n}(x)\to\psi_*(x)$. This proves (6) and also  that $\psi_*\in { L^p(\Omega;\Bbb C)}$ for all $p\geq 1$.

We can prove that $\psi_{\varepsilon_n}\to \psi_*$ in { $L^p(\Omega;\Bbb C)$} by repeating the argument used in the proof of Proposition~\ref{prop:min-seq} (Step~4). In particular, we now know that
\begin{equation}\label{eq:conv:Lp-ms}
\lim_{n\to+\infty}\int_\Omega|\psi|^p\,dx=\int_\Omega|\psi_*|^p\,dx\,.
\end{equation}

Now we prove that $\psi_*\in { H^1_{h\Fab}(\Omega;\Bbb C)}$.   Note that $\Fb_{\varepsilon_n}\to\Fab$ in { $L^q(\Omega;\Bbb R^2)$} for all $q\in[1,2)$; moreover, by \eqref{eq:min-GL-e}, we deduce that
$$
(\nabla-ih\Fb_{\varepsilon_n})\psi_{\varepsilon_n}-(\nabla-ih\Fab)\psi_{\varepsilon_n}=-ih(\Fb_{\varepsilon_n}-\Fab)\psi_{\varepsilon_n}\to0\quad{\rm in~} { L^q(\Omega;\Bbb C^2)}\,.
$$
Since $g$ is the weak limit  of $(\nabla-i\Fb_{\varepsilon_n})\psi_{\varepsilon_n}$ in { $L^2(\Omega;\Bbb C^2)$}, we deduce the following convergence in the distributional sense,
$$
(\nabla-ih\Fab)\psi_{\varepsilon_n}\to g~{\rm in~}{ \mathcal D'(\Omega\setminus\{0\};\Bbb C^2)}\,.
$$
Pick an arbitrary test function { $\varphi\in C_c^\infty(\Omega\setminus\{0\};\Bbb C)$}. By H\"older's inequality,
$$|\langle \psi_{\varepsilon_n}-\psi_*,(\nabla-ih\Fab)\varphi\rangle|\leq \|\psi_{n}-\psi_*\|_{L^p(\Omega)}\|(\nabla-i\Fab)\varphi\|_{L^q(\Omega)}$$
with $\frac1p+\frac1q=1$, $p\in(2,+\infty)$ and $q\in(1,2)$; this proves that $(\nabla-ih\Fab)\psi_{\varepsilon_n}\to(\nabla-ih\Fab)\psi_*$ in { $\mathcal D'(\Omega\setminus\{0\};\Bbb C^2)$}; consequently, $(\nabla-ih\Fab)\psi_*=g\in { L^2(\Omega;\Bbb C^2)}$, which proves (8).

So far we proved the statements (1)-(8) of Proposition~\ref{prop:min-MS*}; it remains to prove the statement (9). For $n\in I$ sufficiently large and $r$ sufficiently small, $\Fb_{\varepsilon_n}=\Fab$ in $\Omega_r$; hence
$$ \int_{\Omega}|(\nabla-i(\ab_{\varepsilon_n}+\Fb_{\varepsilon_n}))\psi_{\varepsilon_n}|^2\,dx\geq \int_{\Omega_r}|(\nabla-i(\ab_{\varepsilon_n}+\Fab))\psi_{\varepsilon_n}|^2\,dx\,.$$
As a consequence of the foregoing inequality, we get
$$\lim_{n\to\infty}\int_{\Omega_r}|(\nabla-i(\ab_{\varepsilon_n}+\Fab))\psi_{\varepsilon_n}|^2\,dx=\int_{\Omega_r}|(\nabla-i(\ab_{*}+\Fab))\psi_{*}|^2\,dx\,.$$
By monotone convergence, we get further
$$\lim_{r\to0_+}\int_{\Omega_r}|(\nabla-i(\ab_{*}+\Fab))\psi_{*}|^2\,dx=\int_{\Omega}|(\nabla-i(\ab_{*}+\Fab))\psi_{*}|^2\,dx\,.$$
This argument  yields that
\begin{equation}\label{eq:mg-e-AB}
\liminf_{n\to+\infty}\int_{\Omega}|(\nabla-i(\ab_{\varepsilon_n}+\Fb_{\varepsilon_n}))\psi_{\varepsilon_n}|^2\,dx\geq \int_{\Omega}|(\nabla-i(\ab_{*}+\Fab))\psi_{*}|^2\,dx\,.\end{equation}
Collecting \eqref{eq:conv:Lp-ms},   \eqref{eq:mg-e-AB} and the convergence established in (5), we finish the proof of (9).
\end{proof}

\section{Proof of the main theorems}\label{sec:thmAB}

In this section, we prove our main Theorems~\ref{thm:ms},
\ref{thm:AB-ns} and \ref{thm:ev-AB}.

\begin{proof}[Proof of Theorem~\ref{thm:ms}]~

\paragraph{\bf Lower bound}

It follows from Proposition~\ref{prop:min-MS*} that
\begin{equation}\label{eq:AB-MS-lb}
\liminf_{\varepsilon\to0_+}\Ee(h)\geq \Eab(\psi_*,\ab_*+\Fab)\geq \Ea(h)\,.
\end{equation}

\paragraph{\bf Upper bound.}\medskip

\paragraph{\it The non-degenerate case}

Assume that $h\not\in 2\pi\Z$ and consider a minimizing configuration $(\psi,\ab)$ of $\Eab$. By Proposition~\ref{prop:min-AB-H1}, $\psi\in H^1(\Omega;\C)$ and $\Fab\psi\in { L^2(\Omega;\Bbb C^2)}$. By dominated convergence,
\begin{equation}\label{eq:min-AB-H1}
\lim_{\varepsilon\to0_+}\int_{D(0,\varepsilon)}|\nabla\psi-ih\ab\psi|^2\,dx=0\quad{\rm and}\quad \lim_{\varepsilon\to0_+}\int_{D(0,\varepsilon)}|\Fab\psi|^2\,dx=0\,.
\end{equation}
For all $\varepsilon\in(0,\varepsilon_0]$, $|\Fb_\varepsilon|\leq |\Fab |$ in $D(0,\varepsilon)$, hence, by  \eqref{eq:min-AB-H1},
\begin{equation}\label{eq:min-AB-H1*}
\lim_{\varepsilon\to0_+}\int_{D(0,\varepsilon)}|\Fb_\varepsilon\psi|^2\,dx =0\,.
\end{equation}
Furthermore, for all $\varepsilon\in(0,\varepsilon_0]$,
\begin{align*}
\Ea(h)=\Eab(\psi,\ab+\Fab)&\geq \Eef(\psi,\ab+\Fb_\varepsilon)-\int_{D(0,\varepsilon)}|(\nabla-ih(\ab+\Fb_\varepsilon))\psi|^2\,dx\\
&\geq \Ee(h)-\int_{D(0,\varepsilon)}|(\nabla-ih(\ab+\Fb_\varepsilon))\psi|^2\,dx\,.
\end{align*}
 Using \eqref{eq:min-AB-H1} and \eqref{eq:min-AB-H1*}, we get that
 $$\lim_{\varepsilon\to0_+}\int_{D(0,\varepsilon)}|(\nabla-ih(\ab+\Fb_\varepsilon))\psi|^2\,dx=0$$
 and consequently
 \begin{equation}\label{eq:Ea>Ee}
 \Ea(h)\geq \limsup_{\varepsilon\to0_+}\Ee(h)\,.
 \end{equation}
Combining this and \eqref{eq:AB-MS-lb}, we get that $(\psi_*,\ab_*)$ is a minimizing configuration of $\Eab$.\medskip

\paragraph{\it The degenerate case}

It remains to prove the inequality \eqref{eq:Ea>Ee} when $h=2\pi n_0$ and $n_0\in\Z$. We introduce the  test function defined in polar coordinates as follows
\begin{equation}\label{eq:t-state-w-ep}
{ w_\var}=\chi_{\varepsilon,p}(r) u(\theta)
\end{equation}
where
$$
p\in(0,1)\,,\quad\chi_{\varepsilon,p}(r)=\begin{cases}\left(\frac{r}{\sqrt{\varepsilon}}\right)^p&{\rm ~if~}0<r<\sqrt{\varepsilon}\\
1&{~\rm if~}r\geq \sqrt{\varepsilon}\end{cases}\quad{\rm and}\quad  u(\theta)=e^{in_0\theta}\,.
$$
 Clearly, ${ w_\var}\in H^1(\Omega;\C)$ and
$$
\lim_{\varepsilon\to0_+}\int_{\Omega}|{ w_\var}|^2\,dx=\lim_{\varepsilon\to0_+}\int_{\Omega}|{ w_\var}|^4\,dx=|\Omega|\,.
$$
Knowing that $(\nabla-i\Fab)u=0$ (see the proof of Proposition~\ref{prop:deg-case}), and that $\Fab=\Fb_\varepsilon$ in $\Omega\setminus D(0,\varepsilon)$, we get
\begin{align*}
&\int_{D(0,\sqrt{\varepsilon})} |(\nabla-ih\Fb_\varepsilon){ w_\var}|^2\,dx\\
&\hskip1cm=\int_0^{2\pi}\left(\int_0^{\sqrt{\varepsilon}} |\partial_r{ w_\var}|^2rdr+\int_\varepsilon^{\sqrt{\varepsilon}}
\frac1{r}
\left|\left(\partial_\theta-in_0\right){ w_\var}\right|^2dr+ \int_0^{\varepsilon}\frac1{r}\left|\left(\partial_\theta-i\frac{n_0r^2}{\varepsilon^2}\right){ w_\var}\right|^2dr\right)d\theta\\
&\hskip1cm=2\pi\int_0^{\sqrt{\varepsilon}}|\chi_{\varepsilon,p}'(r)|^2rdr+2\pi\int_0^{\varepsilon}\frac{n_0^2}{r}\left(1-\frac{r^2}{\varepsilon^2}\right)^2|\chi_{\varepsilon,p}(r)|^2dr\\
&\hskip1cm\leq \pi\left(p+\frac{n_0^2}{p}\varepsilon^{p}\right)\,.
\end{align*}
Writing $\Ee(h)\leq \Eef(w_\varepsilon, \Fb_\varepsilon)$ then taking the limit as $\varepsilon\to0_+$, we infer from the foregoing considerations that
$$\limsup_{\varepsilon\to0_+}\Ee(h)\leq \pi p-\frac{\kappa^2}2|\Omega|\,.$$
Now we send $p$ to $0$ and get
$$\limsup_{\varepsilon\to0_+}\Ee(h)\leq -\frac{\kappa^2}{2}|\Omega|,
$$
which yields the inequality in \eqref{eq:Ea>Ee} in the case $h\in2\pi\Z$, thanks  to Proposition~\ref{prop:deg-case}.\medskip

\paragraph{\bf End of the proof} Having proved that $\lim\limits_{\varepsilon\to0_+}\Ee(h)=\Ea(h)$, we get item (1) in Theorem~\ref{thm:ms}; the statements in item (3)  follow from Proposition~\ref{prop:deg-case}.

Now we need to prove the statement in  item (2). Let $h>0$; we will  prove that $\Ea(h+2\pi)=\Ea(h)$. Indeed, let $(\psi,\ab)_h$ be a minimizing configuration of the functional $\Eab$; it is easy to check that $(\psi_{\rm new},\ab_{\rm new}):=(e^{i\theta}\psi,\frac{h}{h+2\pi}\ab)$ is a minimizing configuration of $\mathcal E_{\rm AB_{h+2\pi}}$, since the expression of the Laplacian in polar coordinates yields
$\nabla (e^{i\theta})=2\pi\Fab e^{i\theta}$.
\end{proof}

\begin{proof}[Proof of Theorem~\ref{thm:AB-ns}]~

Let $(\psi,\ab)_{\kappa,h}$ be a critical configuration of $\Eab$, i.e. a solution of \eqref{eq:GLeq}. We denote by $\|\cdot\|_p$ the usual norm in $L^p(\Omega)$. Expanding the term $\|(\nabla-ih(\ab+\Fab))\psi\|^2_2$, we find
$$\|(\nabla-ih(\ab+\Fab))\psi\|^2_2=\|(\nabla-ih\Fab)\psi\|^2-h^2\|\ab\psi\|^2_2-2h\langle \mathbf j,\mathbf a\rangle\,,$$
where
$$\mathbf j={\rm Im}\big(\overline{\psi}\,(\nabla-ih(\ab+\Fab))\psi\big)=-h\nabla^\bot \curl\ab$$
by the second equation in \eqref{eq:GLeq}. Consequently, by integration by parts,
 $$\langle \mathbf j,\mathbf a\rangle=-h\langle \nabla^\bot\curl\ab,\ab\rangle=h\langle \curl\ab,\curl\ab\rangle=h\|\curl\ab\|^2_2\,.$$
Therefore, after introducing the  energy
$$\mathcal E_0(\psi,\ab+\Fab)= \|(\nabla-ih(\ab+\Fab))\psi\|^2_2-\kappa^2\|\psi\|^2_2+\frac{\kappa^2}{2}\|\psi\|_4^4$$ we get the useful identity
\begin{equation}\label{eq:en-id-new}
\mathcal E_0(\psi,\ab+\Fab)=\mathcal E_0(\psi,\Fab)-h^2\|\ab\psi\|^2_2-2h^2\|\curl\ab\|_2^2\,.
\end{equation}
By the first equation in \eqref{eq:GLeq}, $\mathcal E_0(\psi,\ab+\Fab)=-\frac{\kappa^2}{2}\|\psi\|^4_4\leq 0$, hence we infer from \eqref{eq:en-id-new} the following estimate
\begin{equation}\label{eq:est-E0-new}
0\geq \left(\lambda_{\rm AB}(h)-\kappa^2\right)\|\psi\|_2^2-
h^2\|\ab\|_4^2\|\psi\|_4^2-2h^2\|\curl\ab\|^2_2\,,
\end{equation}
after applying the min-max principle   and the H\"older inequality to estimate the terms $\|(\nabla-ih\Fab)\psi\|_2^2$ and $\|\ab\psi\|_2^2$ respectively.

We estimate the term $\|\curl\ab\|_2^2$ using the second equation in \eqref{eq:GLeq}  as follows,
\[h^2\|\nabla\curl\ab\|_2^2= \big\|\overline{\psi}\,(\nabla-ih(\ab+\Fab))\psi\big\|^2_2\leq \|\psi\|_\infty^2\|(\nabla-ih(\ab+\Fab))\psi\|_2^2\leq \kappa^2 \|\psi\|_2^2\,,\]
after using the estimates (3)-(4) in Proposition~\ref{prop:EL-eq}.  Since $\curl\ab=0$ on $\partial\Omega$, we get further
\[h^2\lambda^D(\Omega)\|\curl\ab\|_2^2\leq  \kappa^2\|\psi\|_2^2\,,\]
where $\lambda^D(\Omega)$ is introduced in \eqref{eq:m*}. As for the term $\|\ab\|_4^2$, we use the forgoing inequality and the constant $m_*(\Omega)$ in \eqref{eq:m*}; we obtain
\[h^2m_*(\Omega)\|\ab\|_4^2\leq h^2\|\curl\ab\|_2^2\leq \frac{\kappa^2}{\lambda^D(\Omega)}\|\psi\|_2^2\,.\]
Plugging  the two forgoing inequalities  into \eqref{eq:est-E0-new} and using $\|\psi\|_4^2\leq \|\psi\|_\infty^2|\Omega|^{1/2}\leq |\Omega|^{1/2}$,  we get\Bk
\[0\geq \left(\lambda_{\rm AB}(h)-\kappa^2-C_*(\Omega)\kappa^2\right)\|\psi\|_2^2\,,
\]
where  $C_*(\Omega)$ is introduced in \eqref{eq:C*}.
This yields that $\|\psi\|_2^2=0$ when $\kappa$ and $h$ satisfy the relation
$\kappa^2<(1+C_*(\Omega))^{-1}\lambda_{\rm AB}(h)$.
\end{proof}
\medskip
\begin{proof}[Proof of Theorem~\ref{thm:ev-AB}]~

{\it Step~1.}

Choose $\varepsilon_0\in(0,1)$  such that  $\overline{D(0,\varepsilon_0)}\subset\Omega$. For all $\varepsilon\in(0,\varepsilon_0)$, we introduce the auxiliary eigenvalue, $\lambda_\varepsilon(h,\Omega)$, in the perforated domain $\Omega_\varepsilon:=\Omega\setminus \overline{D(0,\varepsilon)}$, defined as follows,
\begin{equation}\label{eq:ev-l-ep}
\lambda_\varepsilon(h,\Omega)=\inf_{u\in H^1_\varepsilon(\Omega_\varepsilon)}\frac{\|(\nabla-ih\Fab)u\|_{L^2(\Omega_\varepsilon)}^2}{\|u\|_{L^2(\Omega_\varepsilon)}^2}\,.
\end{equation}
Note that the  circulation \Bk of $h\Fab$ around the hole $D(0,\varepsilon)$ is
$$\Phi_0:=\frac1{2\pi}\int_{ \partial D(0,\varepsilon)}\Fab(x)\cdot dx=\frac{h}{2\pi}\,.$$
By \cite[Thm.~1.1]{HOOO}, for $\varepsilon\in(0,\varepsilon_0)$, $h\geq 0$ and $k\in\Z$,  it holds the following,
\begin{equation}\label{eq:HOOO}
 \lambda_\varepsilon(0,\Omega)=0\leq \lambda_\varepsilon(h,\Omega)\leq \lambda_\varepsilon(\pi,\Omega)~{\rm and}~
\lambda_\varepsilon(h+2k\pi,\Omega)=\lambda_\varepsilon(h,\Omega)\,.
\end{equation}
Furthermore, if $h\in2\pi\Z$, the function $u_0=e^{i\frac{h}{2\pi}\theta}$ is a zero mode for the operator $-(\nabla-ih\Fab)^2$, hence
\begin{equation}\label{eq:HOOO*}
\forall\,h\in 2\pi\Z\,,\quad \lambda_\varepsilon(h,\Omega)=0=\lambda_{\rm AB}(h,\Omega)\,.
\end{equation}
{ We shall show that
\begin{equation}\label{eq5.10}
\lim_{\varepsilon\to 0_+}\lambda_\varepsilon(h,\Omega)=\lambda_{\rm AB}(h,\Omega)\,, \q \forall h\geq 0.
\end{equation}
Then the $2\pi$-periodicity of $\lam_{\rm AB}(h,\O)$ in $h$ follows from \eqref{eq:HOOO}.}

{\it Step~2.}

Denote by $u\in { H^1_{h\Fab}(\Omega;\Bbb C)}$ a normalized ground state of $\lambda_{\rm AB}(h,\Omega)$.  By the min-max principle,
$$\lambda_{\rm AB}(h,\Omega)\geq \|(\nabla-ih\Fab)u\|^2_{L^2(\Omega_\varepsilon)}\geq \lambda_\varepsilon(h,\Omega)\|u\|_{L^2(\Omega_\varepsilon)}^2\,.$$
 By dominated convergence,
\[\lim_{\varepsilon\to0_+}\|u\|_{L^2(\Omega_\varepsilon)}^2=
\|u\|^2_{L^2(\Omega)}=1\,.\]
Consequently,
\begin{equation}\label{eq:l-AB>l-e}\lambda_{\rm AB}(h,\Omega)\geq \limsup_{\varepsilon\to0_+}\lambda_\varepsilon(h,\Omega)\,.
\end{equation}\Bk

{\it Step~3.}

Assume that $0<r<\varepsilon_0$ and $h\not\in2\pi\Z$.  For all $\varepsilon\in(0,r)$,
denote by $u_\varepsilon\in { H^1(\Omega_\varepsilon;\Bbb C)}$ a normalized ground state of $\lambda_\varepsilon(h,\Omega)$; the eigenvalue equation $-(\nabla-ih\Fab)^2u_\varepsilon=\lambda_\varepsilon(h,\Omega)u_\varepsilon$ yields that
\begin{equation}\label{eq:u-eps-bdd-H2}
\|u_\varepsilon\|_{H^2(\Omega_r)}\leq C_r
\end{equation}
for some constant $C_r$ independent from $\varepsilon$.
 Actually, by \eqref{eq:l-AB>l-e}, $\lambda_\varepsilon(h,\Omega)$ is bounded independently of $\varepsilon$, $\Fab\in C^\infty(\overline{\Omega_{r/2}};\R^2)$ and $u_\varepsilon$ satisfies the boundary condition $\nu\cdot\nabla u_\varepsilon =ih\nu\cdot\Fab u_\varepsilon$ on $\partial\Omega$, so we can use the standard elliptic estimates and write
\[\|u_\varepsilon\|_{H^2(\Omega_r)}\leq \hat C_r\big(\|\Delta u_\varepsilon\|_{L^2(\Omega)}+\|u_\varepsilon\|_{L^2(\Omega)}\big)\,,\]
which yields \eqref{eq:u-eps-bdd-H2}.
\Bk  By a diagonal sequence argument, we can extract a sequence $(u_{\varepsilon_n})_{n\geq 1}$ and a function $u_*:\Omega\setminus\{0\}\to\C$ such that $(u_{\varepsilon_n})_{n\geq 1}$ converges to $u_*$ in $H^1(\Omega_{r})$ for  $0<r<\varepsilon_0$.

It then results the following two inequalities,
$$\|u_*\|_{L^2(\Omega_r)}^2= \lim_{n\to+\infty}\|u_{\varepsilon_n}\|_{L^2(\Omega_r)}^2\leq 1\,,$$
and
$$
\liminf_{\varepsilon\to0_+}\lambda_\varepsilon(h,\Omega)\geq \liminf_{n\to+\infty}\lambda_{\varepsilon_n}(h,\Omega)\geq \liminf_{n\to+\infty}\|(\nabla-ih\Fab)u_{\varepsilon_n}\|_{L^2(\Omega_r)}^2
=\|(\nabla-ih\Fab)u_*\|_{L^2(\Omega_r)}^2\,.$$
Sending $r$ to $0$ and using monotone convergence, we deduce that $\|u_*\|_{L^2(\Omega)}\leq 1$, $u_*\in H^1_{\rm h\Fab}(\Omega)$ and
\begin{equation}\label{eq:l-e>l-AB}
\liminf_{\varepsilon\to0_+}\lambda_\varepsilon(h,\Omega)\geq\|(\nabla-ih\Fab)u_*\|_{L^2(\Omega)}^2\geq \lambda_{\rm AB}(h,\Omega)\|u_*\|_{L^2(\Omega)}^2\,.
\end{equation}
Let us prove that
\begin{equation}\label{eq:non-conc}
\limsup_{n\to+\infty}\left(\Be\frac1{r^2}\Bk\|u_{\varepsilon_n}\|^2_{L^2(D(0,r)\setminus D(0,\varepsilon_n))}\right)<+\infty\,.
\end{equation}
The inequality in \eqref{eq:non-conc} results immediately from \eqref{eq:l-AB>l-e} and Lemma~\ref{lem:1D-op}; in fact
$$\left(\frac{h\alpha(h)}{2\pi r}\right)^2\|u_{\varepsilon_n}\|^2_{L^2(D(0,r)\setminus D(0,\varepsilon))}\leq \|(\nabla-ih\Fab)u_{\varepsilon_n}\|^2_{L^2(D(0,r)\setminus D(0,\varepsilon))}\leq \lambda_\varepsilon(h)\,.$$
It now results from \eqref{eq:non-conc}
$$\|u_{\varepsilon_n}\|_{L^2(\Omega_r)}^2=1-\|u_{\varepsilon}\|^2_{L^2(D(0,r)\setminus D(0,\varepsilon_n))}\geq 1-\tilde Cr^2\,,$$
where $\tilde C>0$ is independent from $r$ and $\varepsilon_n$; consequently,
$$\|u_*\|^2_{L^2(\Omega_r)}=\lim_{n\to+\infty}\|u_{\varepsilon_n}\|_{L^2(\Omega_r)}^2\geq 1-\tilde Cr^2\,.$$
Sending $r$ to $0$, we get by monotone convergence that $\|u_*\|_{L^2(\Omega)}^2\geq 1$. Inserting this into \eqref{eq:l-e>l-AB}, we get
$$\lim_{\varepsilon\to0_+}\lambda_\varepsilon(h,\Omega)\geq \lambda_{\rm AB}(h,\Omega)\quad (h\not\in2\pi\Z)\,.$$
Collecting this inequality, \eqref{eq:l-AB>l-e} and \eqref{eq:HOOO*}, we get { \eqref{eq5.10} for $h\not\in 2\pi\Bbb Z$.}

{\it Step~4.}

Let us prove that $\lambda_{\rm AB}(h,\Omega)>0$ for $h\in(0,2\pi)$. Choose $r>0$ so that $D(0,r)\subset\Omega$. Suppose that $\lambda_{\rm AB}(h)=0$. By the min-max principle and \cite[Prop. 2.1]{KP},  we write ($u\in H^1_{h\Fab}(\Omega;\mathbb C)$ is the normalized ground state of $\lambda_{\rm AB}(h,\Omega)$)\Bk
$$ 0=\lambda_{\rm AB}(h,\Omega)\geq \lambda_{\rm AB}(h,D(0,r))\int_{D(0,r)}|u|^2\,dx\geq \frac{\alpha(h)}{4r^2}\int_{D(0,r)}|u|^2\,dx\,,$$
with $\alpha(h)>0$ for $h\in(0,2\pi)$, see \eqref{eq:alpha(h)}. Hence, $u=0$ on $D(0,r)$; by the min-max principle and the diamagnetic inequality,
$$0= \lambda_{\rm AB}(h,\Omega)\geq \lambda^{N,D}(\Omega\setminus D(0,r)) \int_{\Omega\setminus D(0,r)}|u|^2\,dx\,,$$
where $\lambda^{N,D}(\Omega\setminus D(0,r))>0$ is the eigenvalue of the Laplace operator with Neumann condition on $\partial\Omega$ and Dirichlet condition on $\partial D(0,r)$. This proves that $u=0$ on $\Omega\setminus D(0,r)$ too and contradicts the fact that $u$ is a normalized eigenfunction of
$\lambda_{\rm AB}(h,\Omega)$.\medskip

{\it Step~5.}

{ As mentioned above, that $\lambda_{\rm AB}(h)$ is $2\pi$-periodic follows from \eqref{eq:HOOO} and \eqref{eq5.10}.} By Step~4, $\lambda_{\rm AB}(h,\Omega)=0$ iff $h\in2\pi\Z$.

To finish the proof of Theorem~\ref{thm:ev-AB}, it remains to show that $\lambda_{\rm AB}(h,\Omega)$ is continuous at every $h_0\in[0,2\pi)$. If $h_0\in(0,2\pi)$, the result is a simple application of the min-max principle, since, in some neighborhood $I_0\subset(0,2\pi)$ of $h_0$, the space { $ H^1_{h\Fab}(\Omega;\Bbb C)$ is simply   $\mathcal D_0=H^1(\Omega;\Bbb C)\cap L^2(\Omega;\Bbb C;|\Fab|^2\,dx)$} (see Remark~\ref{rem:ev-R2}).  Consequently, the following inequality\footnote{ It results from Cauchy's inequality $2|ab|\leq \delta a+\delta^{-1}b$ with $\delta=|h-h_0|$, $a=(\nabla-ih_0\Fab)u$ and $b=(h-h_0)\Fab u$.}
\begin{multline*}
\left| \int_\Omega |(\nabla-ih\Fab)u|^2\,dx -\int_\Omega|(\nabla-ih_0\Fab)u|^2\,dx\right|\\
\leq 
 |h-h_0|^{1/2} 
\int_\Omega \big(  |(\nabla-ih_0\Fab)u|^2+2|h-h_0|\, |\Fab u|^2\big)\,dx\quad\big(h\in I_0, ~u\mathcal D_0\big)\,,
\end{multline*}
yields\footnote{$\|\Fab u\|^2_{L^2(\Omega)}\leq C_{h_0}\|u\|^2_{L^2(\Omega)}+\|(\nabla-ih_0\Fab)u\|_{L^2(\Omega)}^2$, by Remark~\ref{rem:ev-R2}.}  $ \lim\limits_{h\to h_0}\lambda_{\rm AB}(h,\Omega)=\lambda_{\rm AB}(h_0,\Omega)$. \Bk

So we treat the case $h_0=0$; we would like to prove that
\begin{equation}\label{eq:l-AB-h=0}
\lim_{h\to0_+}\lambda_{\rm AB}(h,\Omega)=0\,.
\end{equation}
 For all $\varepsilon\in(0,\varepsilon_0)$ and $p\in(0,1)$, we introduce the  following quasi-mode
$$w_\varepsilon(x)=\begin{cases} 1 \Bk&{~\rm if~}|x|\geq\sqrt{\varepsilon}\medskip\\
\left(\frac{|x|}{\sqrt{\varepsilon}}\right)^p&~{\rm if~}|x|<\sqrt{\varepsilon}
\end{cases}\,.
$$
The min-max principle and a straight forward computation yields
$$
0<\lambda_{\rm AB}(h)\leq  \frac{\|(\nabla-ih\Fab)w_\varepsilon\|_{L^2(\Omega)}^2}{\|w_
\varepsilon\|_{L^2(\Omega)}^2}\leq C_0( p+ h^2+\varepsilon)\,,
$$
for some constant $C_0$ independent of $h\in (0,2\pi)$; sending $\varepsilon$, $p$ and $h$ to $0_+$, we get \eqref{eq5.10}.
\end{proof}

\begin{proof}[Proof of Theorem~\ref{thm:osi-ms}]\

\noindent{\bf Step 1.}

The non-monotonicity of the function $h\mapsto\lambda(h\Fb_\varepsilon,\Omega)$ results from Theorem~\ref{thm:ev-AB} and
\begin{equation}\label{eq:ev-AB-like}
\lim_{\varepsilon\to0_+}\lambda(h\Fb_\varepsilon,\Omega)=\lambda_{\rm AB}(h,\Omega)\,.
\end{equation}
{\bf Step~1.1.}
We prove the inequality
\begin{equation}\label{eq:ev-AB-like-up}
\limsup_{\varepsilon\to0_+}\lambda(h\Fb_\varepsilon,\Omega)\leq \lambda_{\rm AB}(h,\Omega)\,.
\end{equation}
We start with the case where 
$h\not\in2\pi\Z$. The magnetic Sobolev space $H^1_{h\Fab}(\Omega;\C)$ is embedded in  $H^1(\Omega;\C)\cap L^2(\Omega;|\Fab|^2dx)$ by Remark~\ref{rem:ev-R2}. In which case, we can use a normalized  ground state $u$ of $\lambda_{\rm AB}(h,\Omega)$ and write
\[\lambda_\varepsilon(h\Fb_\varepsilon,\Omega)\leq \int_\Omega|(\nabla-ih\Fb_\varepsilon)u|^2\,dx
=\int_{\Omega_\varepsilon}|(\nabla-ih\Fab)u|^2\,dx+\int_{D(0,\varepsilon)}|(\nabla-ih\Fb_\varepsilon)u|^2\,dx\,.\]
We used that  $\Fb_\varepsilon=\Fab$ on $\Omega_\varepsilon=\Omega\setminus\overline{D(0,\varepsilon)}$; also on $D(0,\varepsilon)$, we have the inequality $|\Fb_\varepsilon|\leq |\Fab|$. Thus, by monotone convergence and  Remark~\ref{rem:ev-R2}, we get
\[\lim_{\varepsilon\to0_+}\int_{\Omega_\varepsilon}|(\nabla-ih\Fab)u|^2\,dx=\lim_{\varepsilon\to0_+}\int_{D(0,\varepsilon)}|(\nabla-ih\Fb_\varepsilon)u|^2\,dx=0\,,\]
which allows us to get \eqref{eq:ev-AB-like-up}. 

We still have to  prove \eqref{eq:ev-AB-like-up} in the case where $h\in 2\pi\Z$.  In which  case $\lambda_{\rm AB}(h,\Omega)=0$ by Theorem~\ref{thm:ev-AB}. We use the trial state $w_\varepsilon$ in \eqref{eq:t-state-w-ep} and write
\[\lambda_\varepsilon(h\Fb_\varepsilon,\Omega)\leq \frac{\|(\nabla-ih\Fb_\varepsilon)w_\varepsilon\|_2^2 }{\|w_\varepsilon\|_2^2 }
\leq \frac{ \pi p+(2\pi)^{-2}p^{-1}\varepsilon^p}{|\Omega|(1-\pi\varepsilon)}\,.\]
Taking the successive limits $\varepsilon\to 0_+$ and $p\to0_+$, we get \eqref{eq:ev-AB-like-up}. 

\noindent{\bf Step~1.2.}
We prove the inequality
\begin{equation}\label{eq:ev-AB-like-lb}
\liminf_{\varepsilon\to0_+}\lambda(h\Fb_\varepsilon,\Omega)\geq \lambda_{\rm AB}(h,\Omega)\,.
\end{equation}
This inequality is trivial in the case where $h\in2\pi\Z$, because $
\lambda_{\rm AB}(h,\Omega)=0$ by Theorem~\ref{thm:ev-AB}. So we handle the case where $h\not\in2\pi\Z$. Let $u_\varepsilon$ be a normalized ground state of $\lambda(h\Fb_\varepsilon,\Omega)$. Since $\Fb_\varepsilon=\Fab$ on $\Omega_\varepsilon=\Omega\setminus\overline{D(0,\varepsilon)}$, we  write by the min-max principle
\begin{equation}\label{eq:ev-AB-like-lb*}
\lambda(h\Fb_\varepsilon,\Omega)\geq \lambda_\varepsilon(h,\Omega)\int_{\Omega_\varepsilon}|u_\varepsilon|^2\,dx\,,
\end{equation}
where $\lambda_\varepsilon(h,\Omega)$ is the eigenvalue introduced in \eqref{eq:ev-l-ep}, which satisfies $\lim\limits_{\varepsilon\to0_+}\lambda_\varepsilon(h,\Omega)=\lambda_{\rm AB}(h,\Omega)$. Arguing as in the proof of  Theorem~\ref{thm:ev-AB}, Step~3, we can extract a function $u_*\in H^1_{h\Fab}(\Omega;\C)$, with $L^2$-norm equal to $1$,  and a sequence $(\varepsilon_n)$ that converges to $0$ so that $u_{\varepsilon_n}\to u_*$ in $H^1_{\rm loc}(\Omega\setminus\{0\};\C)$. Consequently, we infer \eqref{eq:ev-AB-like-lb} from \eqref{eq:ev-AB-like-lb*}.\medskip

\noindent{\bf Step~2.}

If $n$ is even and $h_n=\pi n$, then by \eqref{eq:ev-AB-like} and Theorem~\ref{thm:ev-AB}, $\lim\limits_{\varepsilon\to0_+}\lambda(h_n\Fb_\varepsilon,\Omega)=0$. Thus, given $\kappa>0$, we can write $\lambda(h_n\Fb_\varepsilon,\Omega)<\kappa^2$ for $\varepsilon$ sufficiently small. But this yields that the minimizers of the functional in \eqref{eq:GL-e} are non-trivial.

If $n$ is odd, then by \eqref{eq:ev-AB-like} and Theorem~\ref{thm:ev-AB}, $\lim\limits_{\varepsilon\to0_+}\lambda(h_n\Fb_\varepsilon,\Omega)=\lambda_{\rm AB}(\pi,\Omega)>0$. Consequently, given $\kappa$ such that $0<\kappa^2<(1+C_*(\Omega))^{-1}\lambda_{\rm AB}(\pi,\Omega)$, we can write $\kappa^2<(1+C_*(\Omega))^{-1}\lambda(h_n\Fb_\varepsilon,\Omega)$ for $\varepsilon$ sufficiently small. Now, we can repeat the argument of Theorem~\ref{thm:AB-ns} and write, for any critical point $(\psi_\varepsilon,\Ab_\varepsilon=\ab_\varepsilon+\Fb_\varepsilon)_{\kappa,h_n}$,
\begin{align*}0\geq -\frac{\kappa^2}{2}\|\psi_\varepsilon\|_4^4&=\|(\nabla-ih_n\Fb_\varepsilon)\psi_\varepsilon\|^2_2-\kappa^2\|\psi_\varepsilon\|^2_2+\frac{\kappa^2}{2}\|\psi_\varepsilon\|_4^4-h_n^2\|\ab_\varepsilon\psi_\varepsilon\|^2_2-2h_n^2\|\curl\ab_\varepsilon\|_2^2\\
&\geq \big(\lambda(h_n\Fb_\varepsilon,\Omega)-\kappa^2-C_*(\Omega)\kappa^2\big)\|\psi_\varepsilon\|_2^2
\end{align*}
which in turn yields that $\|\psi_\varepsilon\|_2^2=0$.\Bk
\end{proof}

\subsection*{Acknowledgements}
The authors would like to thank the anonymous  referee for the valuable comments, especially the suggestion to write Theorem~\ref{thm:osi-ms}. This work was partially supported by NYU-ECNU JRI Seed Fund for Collaborative Research.
A. Kachmar's research is partially supported by the Lebanese University in the framework of the project ``Analytical and Numerical Aspects of the Ginzburg-Landau Model''.
X.B. Pan was partially supported by the National Natural Science Foundation of China grant no.  11671143, and 11431005.

\v0.2in


\begin{thebibliography}{100}

\bibitem{AT}  R. Adami, A. Teta. On the Aharonov-Bohm Hamiltonian. {\it Lett. Math.
Phys.} {\bf 43}, 43-54 (1998).\Bk

\bibitem{A} W. Assaad. The breakdown of superconductivity in the presence of magnetic steps. {\it arXiv:1903.04847}.


\bibitem{AK} W. Assaad and A. Kachmar. The influence of magnetic steps on bulk superconductivity. {\it Discrete and Continuous Dynamical Systems (A)}, {\bf 36} (12) (2016), { 6623-6643.}

\bibitem{AKP}  W. Assaad, A. Kachmar, M. Persson-Sundqvist. The distribution of superconductivity near a magnetic barrier. {\it Comm. Math. Phys.} {\bf 366} { (1) (2019), } 269-332.

\bibitem{BF}  V. Bonnaillie-No\"el, S. Fournais. Superconductivity in domains with corners. {\it Rev.
 Math. Phys.} {\bf 19} (6) (2007) 607-637.\Bk

\bibitem{BM} J.F. Braschke and M. Melgaard.
The Friedrichs extension of the
Aharonov-Bohm Hamiltonian on a disk. {\it Integral Equations Operator Theory} {\bf 52} (3) (2005), { 419-436.}

\bibitem{CG}  M. Correggi, E. L. Giacomelli. Surface superconductivity in presence of corners,
{\it Rev. Math. Phys.} {\bf 29} (2) (2017) art. no. 1750005.\Bk

\bibitem{CG1}  M. Correggi, E. L. Giacomelli, Effects of corners in surface superconductivity.
 (2019); arXiv:1908.10112.\Bk

\bibitem{DHS}
N. Dombrowski, P.D. Hislop, E. Soccorsi. Edge currents and eigenvalue estimates for magnetic barrier Schr\"odinger operators. {\it Asymptot. Anal.} {\bf 89} { (3-4) (2014),  331-363.}


\bibitem{DS}  L.  Dabrowski, P.  Stovicek. Aharonov-Bohm effect with $\delta$-type interaction.
{\it J. Math. Phys.} {\bf 39}, 47-62 (1998).\Bk

\bibitem{San}  A.  Deleporte,  S. V$\tilde{\mathrm u}$ Ngoc. Uniform spectral asymptotics for semiclassical
wells on phase space loops. 	arXiv:2002.00234 .  \Bk

\bibitem{E}  L. Erd\"os. 
Dia- and paramagnetism for nonhomogeneous magnetic fields. {\it J. Math. Phys.} {\bf  38}, no. 3, (1997), pp. 1289-1317.\Bk

\bibitem{FH1}  S. Fournais, B. Helffer. 
S. Fournais, B. Helffer. Strong diamagnetism for general domains and applications. {\it Ann. Inst. Fourier} {\bf 57}, no. 7, (2007),  pp. 2389-2400.\Bk 

\bibitem{FH-d}  S. Fournais, B. Helffer. On the third critical field in Ginzburg-Landau theory. {\it Comm. Math. Phys.} {\bf 266} { (1) (2006),  153-196.}

\bibitem{FH2} S. Fournais, B. Helffer. On the Ginzburg-Landau critical field in three dimensions. {\it Comm. Pure Appl. Math.} {\bf 62} (2) 215-241, 
(2009).

\bibitem{FH-b} S. Fournais, B. Helffer. {\it Spectral methods in surface superconductivity}. Progress in Nonlinear Differential Equations and Their Applications. Vol. 77. Birkh\"{a}user Boston Inc., Boston, MA, 2010.

\bibitem{FK-jde}  S. Fournais, A. Kachmar. On the transition to the normal phase for superconductors surrounded by normal conductors.
{\it J.  Differential Equations }
{\bf 247} (6), pp. 1637-1672, 2009.\Bk

\bibitem{FP-ball} S. Fournais, M. Persson.  Strong diamagnetism for the ball in three dimensions. {\it Asymptot. Anal.} {\bf 72} (1-2) 77-123, 2011. \Bk

\bibitem{FP} S. Fournais, M. Persson-Sundqvist. Lack of diamagnetism and the Little-Parks effect. {\it
Comm. Math. Phys.} {\bf 337} (1) (2015), { 191-224.}

\bibitem{GP} T. Giorgi, D. Phillips.  The breakdown of superconductivity due to strong fields for the Ginzburg-Landau model. {\it SIAM J. Math. Anal.} {\bf 30} (2) (1999), { 341-359.}


\bibitem{Hetal}
J. Heinonen, P. Koskela, N. Shanmugalingam, J.  Tyson.
{\it Sobolev spaces on metric measure spaces.
An approach based on upper gradients.} New Mathematical Monographs, Vol. 27. Cambridge University Press, Cambridge, 2015.

\bibitem{HOOO} B. Helffer, M. Hoffmann-Ostenhof, T. Hoffmann-Ostenhof, M.P. Owen.
Nodal sets for ground states of Schr\"odinger operators with zero magnetic field in non-simply connected domains. {\it
Comm. Math. Phys.} {\bf 202} { (3) (1999),  629-649.}

\bibitem{HK-LMP}  B. Helffer, A. Kachmar. The density of superconductivity in domains with corners. {\it Lett.  Math. Phys.} {\bf 108} 2169-2187 (2018). \Bk

\bibitem{HK} B. Helffer, A. Kachmar.
Thin domain limit and counterexamples to strong diamagnetism.
{\it arXiv:1905.06152}.

\bibitem{HK-arma} B. Helffer, A. Kachmar. The Ginzburg-Landau functional with vanishing magnetic field. {\it Arch. Ration. Mech. Anal.} {\bf 218} { (1) (2015),  55-122.}



\bibitem{HPRS}
P.D. Hislop, N. Popoff, N. Raymond, N, M. Persson-Sundqvist.
Band functions in the presence of magnetic steps. {\it
Math. Models Methods Appl. Sci.} {\bf 26} { (1) (2016),  161-184.}




\bibitem{KP} A. Kachmar, X.B. Pan. Superconductivity and the Aharonov-Bohm effect. {\it C. R. Acad. Sci. Paris, Ser. I} {\bf 357} (2019), 216-220.

\bibitem{KS} A. Kachmar, M.  P. Sundqvist. Counterexample to strong diamagnetism for the magnetic Robin Laplacian. arXiv:1910.12499  (2019).

\bibitem{LW} A. Laptev, T. Weidl. Hardy inequalities for magnetic Dirichlet forms.
In Mathematical results in quantum mechanics (Prague, 1998), pages 299-
305. Birkh\"auser, Basel, 1999.


\bibitem{L} C. L\`ena. Eigenvalues variations for Aharonov-Bohm operators. {\it J. Math.
Phys.} {\bf 56} (1) (2015), { article no. 011502.}

\bibitem{LL} E.H. Lieb, M. Loss. {\it Analysis}.
{ $2^{nd}$ edition. Graduate Studies in Math., Vol. 14. Amer. Math. Soc., Providence, RI, 2001.}



\bibitem{LP} W.A. Little, R.D. Parks. Observation of quantum periodicity in the
transition temperature of a superconducting cylinder. {\it Phys. Rev. Lett.}
{\bf 9} (1962), 9-12.


\bibitem{LuP} K.N. Lu, X.B. Pan. Estimates of the upper critical field for the
Ginzburg-Landau equations of superconductivity. {\it Physica D}, {\bf 127} (1-2) (1999), 73-104.




\bibitem{PK} X.B. Pan, K.H. Kwek. Schr\"odinger operators with non-degenerately vanishing magnetic fields in bounded domains. {\it Trans. Amer. Math. Soc.} {\bf 354} (10) (2002), 4201-4227.

\bibitem{RP}
J. Reijniers and F. Peeters. Snake orbits and related magnetic edge states. {\it J. Phys.: Condens. Matter} {\bf 12} (2000), { 9771-9786.}
\bibitem{SS-b} E. Sandier, S. Serfaty.
{\it Vortices in the magnetic Ginzburg-Landau model}. Progress in Nonlinear Differential Equations and their Applications, Vol. 70. Birkh\"auser Boston, Inc., Boston, MA, 2007.



\end{thebibliography}
\end{document}